\documentclass{amsart}
\usepackage{amsmath}
\usepackage{amssymb}
\usepackage{amsfonts}
\usepackage{hyperref}
\usepackage{ifthen}

\theoremstyle{remark}

\numberwithin{equation}{section}

\def\1{\hbox{\rm \bf 1}}

\def\pathloc{ \operatorname{Paths}_{\operatorname{loc}} }
\def\tbpath{ \operatorname{tbPaths}_{\operatorname{loc}} }
\def\paths{ \operatorname{Paths} }

\def\dom{\text{dom}}

\newboolean{workmode}
\setboolean{workmode}{false}

\newboolean{sections}
\setboolean{sections}{true}

\input xy
\xyoption{all}

\makeatletter
\def\th@plain{%
  \thm@notefont{}% same as heading font
  \itshape % body font
}
\def\th@definition{%
  \thm@notefont{}% same as heading font
  \normalfont % body font
}
\makeatother

\newcommand{\NN}{{\mathbb{N}}}  %natural numbers
\newcommand{\QQ}{{\mathbb{Q}}}  %rational numbers
\newcommand{\RR}{{\mathbb{R}}}  %real numbers or cartesian space
\renewcommand{\SS}{{\mathbb{S}}} %sphere
\newcommand{\ZZ}{{\mathbb{Z}}}  %integers

\newcommand{\Ad}{{\operatorname{Ad}}}  %adjoint operator
\newcommand{\colim}{{\operatorname{colim}}} %colimit
\newcommand{\Diff}{{\operatorname{Diff}}}  %diffeomorphisms
\newcommand{\DVB}{{\mathbf{DVB}}} %diffeological vector pseudo-bundles
\newcommand{\Hom}{{\operatorname{Hom}}} %Hom
\newcommand{\id}{{\operatorname{id}}}  %identity
\newcommand{\pr}{{\operatorname{pr}}} %projection
\newcommand{\supp}{{\operatorname{supp}}}  %support
\newcommand{\g}{{\mathfrak{g}}} %generic Lie algebra
\newcommand{\GL}{{\operatorname{GL}}}  %general linear
\newcommand{\SO}{{\operatorname{SO}}}  %special orthogonal
\newcommand{\CIN}{{C^\infty}}   %infinitely differentiable
\newcommand{\eps}{\varepsilon}  %var epsilon
\newcommand{\hook}{{\lrcorner\,}} %hook or interior multiplication
\newcommand{\tr}{{\operatorname{tr}}} %translation

\newcommand{\hypref}[2]{{\hyperref[#1]{#2~\ref{#1}}}}

\newcommand{\ifwork}[1]{\ifthenelse{\boolean{workmode}}{#1}{}}
\newcommand{\comment}[1]{}
\newcommand{\mute}[1]{}
\newcommand{\printname}[1]{}

\ifwork{
\renewcommand{\comment}[1]{{\marginpar{*}\ \scriptsize{#1}\ }}
\renewcommand{\printname}[1]
    {\smash{\makebox[0pt]{\hspace{-1.0in}\raisebox{8pt}{\tiny #1}}}}
}

\newcommand{\labell}[1] {\label{#1} \printname{#1}}

\newcommand{\ifsection}[2]{\ifthenelse{\boolean{sections}}{#1}{#2}}

\theoremstyle{plain}
\ifsection{
    \newtheorem{theorem}{Theorem}[section]
}
{
    \newtheorem{theorem}{Theorem}
}

\newtheorem{proposition}[theorem]{Proposition}

\newtheorem{corollary}[theorem]{Corollary}
\newtheorem{lemma}[theorem]{Lemma}

\theoremstyle{definition}
\newtheorem{definition}[theorem]{Definition}
\newtheorem{example}[theorem]{Example}
\newtheorem{remark}[theorem]{Remark}

\newtheorem{question}{Question}

\def\eoe{\unskip\ \hglue0mm\hfill$\diamond$\smallskip\goodbreak} 

\title{The Diffeology of Milnor's Classifying Space}

\author{Jean-Pierre Magnot}
\address{LAREMA, Universit\'e d’Angers, 2 Bd Lavoisier
	, 49045 Angers cedex 1, France and Lyc\'ee Jeanne d'Arc, 40 avenue de Grande Bretagne, 63000 Clermont-Ferrand, France}
\email{jean-pierr.magnot@ac-clermont.fr}

\author{Jordan Watts}
\address{Department of Mathematics, University of Colorado Boulder, Campus Box 395, Boulder, CO, USA 80309.}
\email{jordan.watts@colorado.edu}

\thanks{2010 AMS \emph{Mathematics subject classification}. Primary: 53C05, 57R55; Secondary: 58B05, 58B10.}

\begin{document}

\begin{abstract}
We define a diffeology on the Milnor classifying space of a diffeological group $G$, constructed in a similar fashion to the topological version using an infinite join.  Besides obtaining the expected classification theorem for smooth principal bundles, we prove the existence of a diffeological connection on any principal bundle (with mild conditions on the bundles and groups), and apply the theory to some examples, including some infinite-dimensional groups, as well as irrational tori.  
\end{abstract}

\maketitle

\section{Introduction}\labell{s:intro}

Let $G$ be a topological group with a reasonable topology (Hausdorff, paracompact, and second-countable, say).  Milnor \cite{milnor2} constructed a universal topological bundle $EG\to BG$ with structure group $G$ satisfying:
\begin{itemize}
\item for any principal $G$-bundle $E\to X$ over a space (again, assume Hausdorff, paracompact, and second-countable) there is a continuous ``classifying map'' $F\colon X\to BG$ for which $E$ is $G$-equivariantly homeomorphic to the pullback bundle $F^*EG,$

\item any continuous map $F\colon X\to BG$ induces a bundle $F^*EG$ with $F$ a classifying map, and

\item any two principal $G$-bundles are isomorphic if and only if their classifying maps are homotopic.
\end{itemize}

To push this classification into the realm of Lie groups and smoothness, one would need smooth structures on the spaces $EG$ and $BG$.  Many approaches exist, which extend to infinite-dimensional groups (see, for example, \cite{mostow}, \cite[Theorem 44.24]{KM}).  The main point here is that $EG$ and $BG$ are not typically manifolds, and so a more general smooth structure is required.

In this paper, we take the approach of diffeology.  The language is quite friendly, allowing one to differentiate and apply other analytical tools with ease to infinite-dimensional groups such as diffeomorphism groups, including those of non-compact manifolds, as well as projective limits of groups, including some groups which appear naturally in the ILB setting of Omori \cite{Om}, and may not exhibit atlases.  Moreover, we can include in this paper interesting groups which are not typically considered as topological groups.  For example, irrational tori (see Example~\ref{x:irrational torus}) have trivial topologies (and hence have no atlas) and only constant smooth functions; however, they have rich diffeologies, which hence are ideal structures for studying the groups.  Irrational tori appear in important applications, such as pre-quantum bundles on a manifold associated to non-integral closed 2-forms (see Subsection~\ref{ss:irrational torus}).  Another benefit of using diffeology is that we can directly use the language to construct connection 1-forms on $EG$. Our main source for the preliminaries on diffeology is the book by Iglesias-Zemmour \cite{iglesias}.

Our main results include Theorem~\ref{t:BG}, which states that there is a natural bijection between isomorphism classes of so-called D-numerable principal $G$-bundles over a Hausdorff, second-countable, smoothly paracompact diffeological space $X$, and smooth homotopy classes of maps from $X$ to $BG$.  This holds for any diffeological group $G$.  As well, we prove Theorem~\ref{t:iz-connection}, which states that if $G$ is a regular diffeological Lie group, and $X$ is Hausdorff and smoothly paracompact, then any D-numerable principal $G$-bundle admits a connection; in particular, such a bundle admits horizontal lifts of smooth curves.  These theorems provide a method for constructing classifying spaces different, for example, to what Kriegl and Michor do in \cite[Theorem 44.24]{KM} with $G=\Diff(M)$ for $M$ a compact smooth manifold, where they show that the space of embeddings of $M$ into $\ell^2$ yields a classifying space for $G$.

Our framework is applied to a number of situations.  We show that $EG$ is contractible (Proposition~\ref{p:contractible}), which allows us to study the homotopy of $BG$ (Proposition~\ref{p:homotopy of BG}).  We also study smooth homotopies between groups, and how these are reflected in classifying spaces and principal bundles (Subsection~\ref{ss:homotopies}).  We transfer the theory to general diffeological fibre bundles via their associated principal $G$-bundles (Corollary~\ref{c:BG}) and discuss horizontal lifts in this context (Proposition~\ref{p:assoc lifts}).  We also transfer the theory to diffeological limits of groups, with an application to certain ILB principal bundles (Proposition~\ref{p:ilb}).  We show that a short exact sequence of diffeological groups induces a long exact sequence of diffeological homotopy groups of classifying spaces (Proposition~\ref{p:les}),
and apply this to a short exact sequence of pseudo-differential operators and Fourier integral operators (Example~\ref{x:pdo}).
Finally, we apply this theory to irrational torus bundles, which are of interest in geometric quantisation \cite{weinstein}, \cite{iglesias-bdles}, \cite[Articles 8.40-8.42]{iglesias} and the integration of certain Lie algebroids \cite{crainic}.

This paper is organised as follows.  
Section~\ref{s:diffeology} reviews necessary prerequisites on diffeological groups (including diffeological Lie groups and regular groups), internal tangent bundles, and diffeological fibre bundles.  In Section~\ref{s:classifying sp} we construct the diffeological version of the Milnor classifying space, $EG\to BG$, and prove Theorem~\ref{t:BG}.  In Section~\ref{s:connection}, we introduce the theory of connections from the diffeological point-of-view, and prove Theorem~\ref{t:iz-connection}.  Finally, in Section~\ref{s:applications}, we have our applications.

A few open questions are inspired.  The conditions on the D-topology (in particular, Hausdorff and second-countable conditions, and sometimes smooth paracompactness as well) seem out-of-place in the general theory of diffeology.
Even though these conditions are satisfied in our examples, in some sense, topological conditions and arguments should be \emph{replaced} with diffeological conditions and arguments.  This leads to the first question.
\begin{question}\labell{q:replace top}
Under what conditions is the D-topology of a diffeological space Hausdorff, second-countable, and smoothly paracompact?  Can one weaken these conditions?
\end{question}
An answer to this would allow us to rephrase Theorem~\ref{t:BG} in a manner more natural to diffeology.
A partial answer, for example, is to require that the smooth real-valued functions separate points: in this case, the weakest topology induced by the functions is Hausdorff, and this is a sub-topology of the D-topology.  It follows that the D-topology is Hausdorff.

Another question is an obvious one:
\begin{question}\labell{q:universal bdle}
Given a diffeological group $G$, if $E\to X$ is a principal $G$-bundle satisfying necessary mild conditions, and the total space $E$ is diffeologically contractible, is $E\to X$ \textbf{universal} in the sense that any principal $G$-bundle over a sufficiently nice diffeological space $Y$ has a unique (up to smooth homotopy) classifying map to $X$? (The answer is affirmative in the topological category.)
\end{question}

The theme of universality continues:
\begin{question}\labell{q:universal conn}
Let $G$ be a regular diffeological Lie group (Definition~\ref{d:diffeol lie gp}), and let $E\to X$ be a principal $G$-bundle in which the D-topology on $X$ satisfies mild conditions.  Is every diffeological connection, or even every connection 1-form, the pullback of the diffeological connection (or connection 1-form) constructed on $EG\to BG$ in Theorem~\ref{t:universal connection}?
\end{question}

Some of these questions have been addressed in \cite{CW2}.  As mentioned above, Kriegl and Michor give a classifying space for the group $\Diff(M)$ where $M$ is a compact smooth manifold, and also show that this classifying space has a universal connection.  However, the proof uses the fact that every such $M$ has an isometric embedding into some Euclidean space, and we no longer have this advantage for general diffeological groups.

\subsection*{Acknowledgements}
The authors would like to thank Patrick Iglesias-Zemmour and the other organisers of the ``Workshop on Diffeology, etc.'', held in Aix-en-Provence in June 2014, where the discussions that lead to this paper began.  We would also like to thank Daniel Christensen and Enxin Wu, as well as the anonymous referee, for excellent comments and suggestions.

%%%%%%%%%%%%%%%%%%%%%%%%%%%%%%%%%%%%%%%%%%%%%%%%%%%%%%%%%%%%%%%%%%%%%%%%%%
\section{Preliminaries}\labell{s:diffeology}
%%%%%%%%%%%%%%%%%%%%%%%%%%%%%%%%%%%%%%%%%%%%%%%%%%%%%%%%%%%%%%%%%%%%%%%%%%

This section provides background on diffeological groups and fibre bundles.  For a review of more basic properties of diffeological spaces, we refer to the book of Iglesias-Zemmour \cite{iglesias}.  In particular, products, sub-objects, quotients, and underlying D-topologies of diffeological spaces, as well as homotopies of maps between them, are used throughout this paper. 

%% %% %%
\subsection{Diffeological Groups and Internal Tangent Bundles}\labell{ss:groups}
%% %% %% 

In this subsection we review diffeological groups and their actions.  We then introduce the internal tangent bundle of a diffeological space with the goal of obtaining a Lie algebra for certain diffeological groups admitting an exponential map.  For more details on the basics of diffeological groups, see \cite[Chapter 7]{iglesias}.  For more on internal tangent bundles, see \cite{CW}.  For more details on diffeological Lie groups, see \cite{leslie} and \cite{KM} (although the latter reference deals with infinite-dimensional groups, not diffeological ones).

\begin{definition}[Diffeological Groups and their Actions]\labell{d:group}
A \textbf{diffeological group} is a group $G$ equipped with a diffeology such that the multiplication map $m\colon G\times G\to G$ and the inversion map $\operatorname{inv}\colon G\to G$ are smooth.  A \textbf{diffeological group action} of $G$ on a diffeological space $X$ is a group action in which the map $G\times X\to X$ sending $(g,x)$ to $g\cdot x$ is smooth.  Fixing $g\in G$, denote by $L_g$ left multiplication by $g$, and by $R_g$ right multiplication by $g$.  Denote by $e$ the identity element of $G$.
\end{definition}

In order to begin associating a Lie algebra to certain diffeological groups, we need to establish the theory of tangent spaces and bundles to a diffeological space.

\begin{definition}[Internal Tangent Space]\labell{d:int tang sp}
Let $(X,\mathcal{D})$ be a diffeological space, and fix $x\in X$.  Let $\mathbf{Plots_0}(\mathcal{D},x)$ be the category whose objects are plots $(p\colon U\to X)\in\mathcal{D}$ for which $U$ is connected, $0\in U$, and $p(0)=x$; and whose arrows from $p\colon U\to X$ to $q\colon V\to X$ are commutative triangles 
$$\xymatrix{
U \ar[rr]^{f} \ar[dr]_{p} & & V \ar[dl]^{q} \\
 & X & \\
}$$
where $f\colon U\to V$ is smooth and $f(0)=0$.  Let $F$ be the forgetful functor from $\mathbf{Plots}_0(\mathcal{D},x)$ to $\mathbf{Open}_0$, the category whose objects are connected open subsets of Euclidean spaces containing $0$ and whose arrows are smooth maps between these that fix $0$.  $F$ sends a plot to its domain, and a commutative triangle as above to $f\colon U\to V$.  Finally, let $\mathbf{Vect}$ be the category of vector spaces with linear maps between them, and let $T_0\colon \mathbf{Open}_0\to\mathbf{Vect}$ be the functor sending $U$ to $T_0U$ and $f\colon U\to V$ to $f_*|_0\colon T_0U\to T_0V$.  Then define the \textbf{internal tangent space} $T_xX$ to be the colimit of the functor $T_0\circ F$.  Denote by $TX$ the set $\bigsqcup_{x\in X}T_xX$.
\end{definition}

In order to equip $TX$ with a suitable diffeology, we need a few more definitions.

\begin{definition}[Induced Tangent Maps]\labell{d:ind tang maps}
Let $(X,\mathcal{D})$ be a diffeological space.  For a plot $(p\colon U\to X)\in\mathbf{Plots}_0(\mathcal{D},x)$, denote by $p_*$ the map sending vectors in $T_0U$ to $T_{p(0)}X$ given by the definition of a colimit.  Extend this to more general plots $p\colon U\to X$ in $\mathcal{D}$ as follows.  Define $p_*\colon TU\to TX$ to be the map that sends $v\in T_uU$ to the element in $T_{p(u)}X$ given by $(p\circ\tr_u)_*((\tr_{-u})_*v)$, where $\tr_u$ is the translation in Euclidean space sending $U-u$ to $U$.  If $c\colon\RR\to X$ is a smooth curve, then by $\frac{dc}{dt}$ we mean $c_*\left(\frac{d}{dt}\right)$ where $\frac{d}{dt}$ is the constant section $\RR\to T\RR\colon t\mapsto(t,1)$.  We may also denote $p_*$ by $Tp$, emphasising the functoriality of $T$ (see Remark~\ref{r:tangent bdle}).
\end{definition}

\begin{definition}[Diffeological Vector Spaces]\labell{d:diffeol vect sp}
A \textbf{diffeological vector space} is a vector space $V$ over $\RR$ equipped with a diffeology such that addition $+\colon V\times V\to V$ and scalar multiplication $\cdot\colon\RR\times V\to V$ are smooth.  %$V$ is a \textbf{fine} diffeological vector space if its diffeology is the smallest one on $V$ such that it is diffeological vector space.
\end{definition}

We use the terminology of \cite{pervova}, but continue to follow \cite{CW}.

\begin{definition}[Diffeological Pseudo-Bundles]\labell{d:pseudobundle}
A \textbf{(diffeological) vector pseudo-bundle} $\pi\colon E\to X$ is a pair of diffeological spaces $E$ and $X$ with a smooth surjection $\pi\colon E\to X$ between them such that for each $x\in X$ the fibre $\pi^{-1}(x)$ is a diffeological vector space.  Moreover, fibrewise addition $+\colon E\times_X E\to E$, scalar multiplication $\cdot\colon\RR\times E\to E$, and the zero section $X\hookrightarrow E$ are required to be smooth.
\end{definition}

\begin{remark}\labell{r:pseudobundle}
Let $\pi\colon E\to X$ be a smooth map between diffeological spaces such that for each $x$, the fibre $\pi^{-1}(x)$ is a diffeological vector space.  Then there is a smallest diffeology on $E$ that contains the original diffeology on $E$ and so that $\pi$ is a vector pseudo-bundle; see \cite[Proposition 4.6]{CW}.
\end{remark}

\begin{definition}[Internal Tangent Bundle]\labell{d:tangent bdle}
Define the \textbf{internal tangent bundle} of a diffeological space $X$ to be the set $TX:=\bigsqcup_{x\in X}T_xX$ equipped with the smallest diffeology $T\mathcal{D}$ such that
\begin{enumerate}
\item $TX$ is a vector pseudo-bundle,
\item for each plot $p\colon U\to X$, the induced map $Tp\colon TU\to TX$ is smooth.
\end{enumerate}
\end{definition}

\begin{remark}\labell{r:tangent bdle}
$T$ is a functor, sending diffeological spaces $(X,\mathcal{D})$ to $(TX,T\mathcal{D})$.  Consequently, we get the chain rule: given smooth maps $f\colon X\to Y$ and $g\colon Y\to Z$, we have that $T(g\circ f)=Tg\circ Tf$.   Moreover, $T$ respects products: let $X$ and $Y$ be diffeological spaces.  Then $T(X\times Y)\cong TX\times TY$ \cite[Proposition 4.13]{CW}.
\end{remark}

We now arrive at an important generalisation of what is very well known in standard Lie group theory.

\begin{theorem}[Tangent Bundle of a Diffeological Group]\labell{t:group}
Let $G$ be a diffeological group with identity element $e$.  Then, $TG$ is isomorphic as a vector pseudo-bundle to $G\times T_eG$.
\end{theorem}

\begin{proof}
This is \cite[Theorem 4.15]{CW}.  The diffeomorphism is given by sending $v\in T_gG$ to $(g,(L_{g^{-1}})_*v)$.
\end{proof}

Diffeological groups admit adjoint actions of $G$ on $T_{e}G$.

\begin{definition}[Adjoint Action]\labell{d:adjoint action}
Let $G$ be a diffeological group, and let $C$ be the (smooth) \textbf{conjugation action}: $C\colon G\times G\to G\colon (g,h)\mapsto ghg^{-1}$.  Define the \textbf{adjoint action} of $G$ on $T_{e}G$ to be the (smooth) action $g\cdot\xi:=TC_g(\xi)$ for all $\xi\in T_eG$.
\end{definition}

This does not immediately give us a well-defined infinitesimal adjoint action (\emph{i.e.} Lie bracket) on $T_{e}G$.  This issue is resolved by Leslie in \cite{leslie} by considering the following type of diffeological group.

\begin{definition}[Diffeological Lie Group]\labell{d:lie gp}
A diffeological group $G$ is a \textbf{diffeological Lie group} if
\begin{enumerate}
\item for every $\xi\in T_eG$ there exists a smooth real-valued linear functional $\ell\colon T_eG\to\RR$ satisfying $\ell(\xi)\neq 0$,
\item the plots of $T_{e}G$ smoothly factor through a map $\varphi\colon V\to T_eG$, where $V$ is an open subset of a complete Hausdorff locally convex topological vector space.
\end{enumerate}
\end{definition}

\begin{remark}[$T_{e}G$ is a Lie Algebra]\labell{r:lie alg}
Let $G$ be a diffeological Lie group.  Then $T_{e}G$ is a Lie algebra under the infinitesimal adjoint action (\cite[Theorem 1.14]{leslie}). In this case, we denote $T_{e}G$ by $\g$.
\end{remark}

Finally, in order to associate curves in $G$ with curves in $\g$, we need to introduce the notion of a \emph{regular} Lie group, following \cite{leslie}, \cite[Section 38]{KM}, and \cite{Om}.  This correspondence is important when obtaining horizontal lifts of fibre bundles from connection 1-forms later in the paper.

\begin{definition}\labell{d:diffeol lie gp}
A diffeological Lie group $G$ is \textbf{regular} if there is a smooth map 
$$\exp\colon \CIN([0,1],\g)\to\CIN([0,1],G)$$
sending a smooth curve $\xi(t)$ to a smooth curve $g(t)$ such that $g(t)$ is the unique solution of the differential equation
$$\begin{cases}
g(0)= e,\\
(R_{g(t)^{-1}})_*\frac{dg(t)}{dt}= \xi(t).
\end{cases}$$
\end{definition}

\begin{remark}\labell{r:[0,1]}
Unless specified otherwise, we take the subset diffeology on $[0,1]\subseteq\RR$ throughout the paper.
\end{remark}

\begin{remark}\labell{r:diffeol lie gp}
Not all diffeological groups are regular.  For example, the diffeological group $\Diff_+(0,1)$ see \cite{Ma2011} or \cite[Subsection 2.5.2]{Ma2015}.
\end{remark}

\begin{remark}\labell{r:reg diffeol lie gp}
If $G$ is a regular diffeological Lie group, then any $v\in\g$ is a germ of a smooth path $c(t)=\exp(tv)$.  This fact is not guaranteed when $G$ is not regular, as in this case $v\in T_{e}G$ is only a sum of germs, and not necessarily a germ on its own.  This explains the different notations and definitions found in \cite{Ma2013} for Fr\"olicher Lie groups.  However, when $G$ is regular and $\g$ complete, the different definitions coincide.
\end{remark}

%% %% %%
\subsection{Diffeological Fibre Bundles}\labell{ss:bundles}
%% %% %%

We now review diffeological fibre bundles and principal $G$-bundles; for more details, see \cite{iglesias}.  Instead of using the original definition of a diffeological fibre bundle given in \cite{IgPhD}, we take the equivalent definition below (see \cite[Article 8.9]{iglesias}).  Similarly, we take an equivalent definition for a principal $G$-bundle (see \cite[Article 8.13]{iglesias}).

\begin{definition}[Diffeological Fibre Bundles]\labell{d:fibre bundles}
Let $\pi\colon E\to X$ be a smooth surjective map of diffeological spaces.
\begin{enumerate}
\item Define the \textbf{pullback} of $\pi\colon E\to X$ by a smooth map $f\colon Y\to X$ of diffeological spaces to be the set $$f^*E\colon =\{(y,e)\in Y\times E~|~f(y)=\pi(e)\}$$ equipped with the subset diffeology induced by the product.  It comes with the two smooth maps $\tilde{f}\colon f^*E\to E$ and $\pi_f\colon f^*E\to Y$ induced by the projection maps, the latter of which is also surjective.

\item $\pi\colon E\to X$ is \textbf{trivial} with \textbf{fibre} $F$ if there is a diffeological space $F$ and a diffeomorphism $\varphi\colon E\to X\times F$ making the following diagram commute.

    $$\xymatrix{
    E \ar[rr]^{\varphi} \ar[dr]_{\pi} & & X\times F \ar[dl]^{\pr_1} \\
     & X & \\
    }$$

\item $\pi\colon E\to X$ is \textbf{locally trivial} with fibre $F$ if there exists an open cover $\{U_\alpha\}_{\alpha}$ in the D-topology on $X$ such that for each $\alpha$, the pullback of $\pi\colon E\to X$ to $U_\alpha$ via the inclusion map is trivial with fibre $F$.  The collection $\{(U_\alpha,\varphi_\alpha)\}_\alpha$, where the diffeomorphism $\varphi_\alpha\colon E|_{U_\alpha}\to U_\alpha\times F$ is as above, make up a \textbf{local trivialisation} of $\pi$.

\item We say that $\pi\colon E\to X$ is a \textbf{diffeological fibration} or \textbf{diffeological fibre bundle} if for every plot $p\colon U\to X$, the pullback bundle $p^*E\to U$ is locally trivial.

\item Let $G$ be a diffeological group.  We say that a diffeological fibre bundle $\pi\colon E\to X$ is a \textbf{principal $G$-bundle} if there is a smooth action of $G$ on $E$ for which every plot $p\colon U\to X$ induces a pullback bundle $p^*E\to U$ that is \textbf{locally equivariantly trivial}.  That is, for any $u\in U$ there exists an open neighbourhood $V$ of $u$, a plot $q\colon V\to E$ satisfying $\pi\circ q=p|_V$, and an equivariant diffeomorphism $\psi\colon V\times G\to p|_V^*E$ sending $(v,g)$ to $(v,g\cdot q(v))$.  Furthermore, a principal $G$-bundle itself is \textbf{locally equivariantly trivial} if it admits a local trivialisation $\{U_\alpha,\varphi_\alpha\}_\alpha$ in which each $\varphi_\alpha$ is $G$-equivariant.

\item We say that a diffeological fibre bundle is \textbf{weakly D-numerable} if there exists a local trivialisation $\{U_i\}_{i\in\NN}$ of $X$ and a pointwise-finite smooth partition of unity $\{\zeta_i\}_{i\in\NN}$ of $X$ such that $\zeta_i^{-1}((0,1])\subseteq U_i$ for each $i$.  We say that a weakly D-numerable diffeological fibre bundle is \textbf{D-numerable} if the local trivialisation can be chosen to be locally finite and the smooth partition of unity subordinate to $\{U_i\}$; that is, $\supp(\zeta_i)\subseteq U_i$ for each $i\in\NN$.
\end{enumerate}
\end{definition}

The point of having weak D-numerability versus D-numerability is that \emph{a priori}, as we will see below, the Milnor construction as a bundle is not necessarily D-numerable, only weakly D-numerable.

\begin{remark}\labell{r:pullback}
Diffeological fibre bundles pull back to diffeological fibre bundles.  Moreover, trivial bundles pull back to trivial bundles. Consequently, (weakly) D-numerable bundles pull back to (weakly) D-numerable bundles.  Similar statements hold for principal $G$-bundles.
\end{remark}

\begin{remark}\labell{r:triviality}
Not every diffeological fibre bundle is locally trivial.  The 2-torus modulo the irrational Kronecker flow is such an example \cite[Article 8.38]{iglesias}.
\end{remark}

\begin{remark}\labell{r:d-numerable}
A weakly D-numerable principal diffeological fibre bundle with Hausdorff, smoothly paracompact base is D-numerable.  A locally trivial diffeological fibre bundle over a Hausdorff, second-countable, smoothly paracompact base is also D-numerable.  We include the definition of smooth paracompactness below for completeness.
\end{remark}

\begin{definition}[Smooth Paracompactness]\labell{d:smooth paracompactness}
A diffeological space $X$ is \textbf{smoothly paracompact} if the D-topology of $X$ is paracompact (any open cover admits a locally finite open refinement), and any open cover of $X$ admits a \emph{smooth} partition of unity subordinate to it.
\end{definition}

%%%%%%%%%%%%%%%%%%%%%%%%%%%%%%%%%%%%%%%%%%%%%%%%%%%%%%%%%%%%%%%%%%%%%%%%%%
\section{Milnor's Classifying Space}\labell{s:classifying sp}
%%%%%%%%%%%%%%%%%%%%%%%%%%%%%%%%%%%%%%%%%%%%%%%%%%%%%%%%%%%%%%%%%%%%%%%%%%

We begin this section with the construction of a classifying space of a diffeological group $G$, completely in the diffeological category (Proposition~\ref{p:EG principal}).  Under certain topological constraints on a diffeological space $X$, we get the desired natural bijection between homotopy classes of diffeologically smooth maps $X\to BG$ and D-numerable principal $G$-bundles over $X$ (see Theorem~\ref{t:BG}).

\subsection{The Milnor Construction}\labell{ss:construction}

We construct Milnor's classifying space in the diffeological category for a diffeological group $G$.  This results in a principal $G$-bundle $\pi\colon EG\to BG$ that is at least weakly D-numerable.

\begin{definition}[Join Operation]\labell{d:join}
Let $\{X_i\}_{i\in\NN}$ be a family of diffeological spaces.  Define the \textbf{join} $\bigstar_{i\in\NN}X_i$ of this family as follows.  Take the subset $S$ of the product $\left(\Pi_{i\in \NN}[0,1]\right)\times\left(\Pi_{i\in \NN}X_i\right)$ consisting of elements $(t_i,x_i)_{i\in \NN}$ in which only finitely many of the $t_i$ are non-zero, and $\sum_{i\in \NN}t_i=1$.  Equip $S$ with the subset diffeology induced by the product diffeology.  Let $\sim$ be the equivalence relation on $S$ given by: $(t_i,x_i)_{i\in \NN}\sim(t'_i,x'_i)_{i\in \NN}$ if
\begin{enumerate}
\item $t_i=t'_i$ for each $i\in \NN$, and
\item if $t_i=t'_i\neq 0$, then $x_i=x_i'$.
\end{enumerate}
Then the join is the quotient $S/\!\sim$ equipped with the quotient diffeology.  We denote elements of $\bigstar_{i\in \NN}X_i$ by $(t_ix_i)$.
\end{definition}

\begin{definition}[Classifying Space]\labell{d:EG}
Let $G$ be a diffeological group.  Define $EG = \bigstar_{i \in \NN}G$.  There is a natural smooth action of $G$ on $EG$ given by $h\cdot(t_ig_i)=(t_ig_ih^{-1})$ induced by the diagonal action of $G$ on $G^{\NN}$ and the trivial action on $[0,1]^{\NN}$.  Denote the quotient $EG/G$ by $BG$ and elements of $BG$ by $[t_ig_i]$.  This is the \textbf{(diffeological) Milnor classifying space} of $G$.  Since we will make use of it, denote by $S_G$ the subset of $[0,1]^\NN\times G^\NN$ corresponding to the set $S$ in Definition~\ref{d:join}.
\end{definition}

\begin{proposition}\labell{p:EG action smooth}
The action of $G$ on $EG$ is smooth.
\end{proposition}

\begin{proof}
We want to show that the map $a_E\colon G\times EG\to EG$ sending $(h,(t_ig_i))$ to $(t_ig_ih^{-1})$ is smooth.  Let $p\colon U\to G\times EG$ be a plot.  It is enough to show that $a_E\circ p$ locally lifts to a plot of $S_G$ via the quotient map $\pi\colon S_G\to EG$.  Let $\pr_i$ ($i=1,2$) be the natural projection maps on $G\times EG$.  Then there exist an open cover $\{U_\alpha\}$ of $U$ and for each $\alpha$ a plot $q_\alpha\colon U_\alpha\to S_G$ such that $$\pr_2\circ p|_{U_\alpha}=\pi\circ q_\alpha.$$  Let $a\colon G\times S_G\to S_G$ be the smooth map sending $(h,(t_i,g_i))$ to $(t_i,g_ih^{-1})$.  Then, $$a_E\circ p|_{U_\alpha}=\pi\circ a(q_\alpha,\pr_1\circ p|_{U_\alpha}),$$ where the right-hand side is a plot of $EG$.
\end{proof}

\begin{proposition}[$EG\to BG$ is Principal]\labell{p:EG principal}
$EG \to BG$ is a weakly D-numerable principal $G$-bundle.
\end{proposition}

\begin{proof}
Define for each $j\in\NN$ the function $s_j\colon EG\to[0,1]$ by $s_j(t_ig_i):=t_j$.  These are diffeologically smooth, hence continuous with respect to the D-topology.  For each $j\in\NN$ define the open set $$V_j:=s_j^{-1}(0,1]=\{(t_ig_i) \mid t_j > 0\}.$$  Since for any point $(t_ig_i)\in EG$ we have that $\sum t_i=1$, it follows that $\{V_j\}$ forms an open cover of $EG$.  Since each $V_j$ is $G$-invariant, setting $U_j:= \pi(V_j)$ we get an open covering $\{U_j\}$ of $BG$.

Define for each $j\in \NN$ the map $\varphi_j\colon V_j \to G\times U_j$ by $$ \varphi_j(t_ig_i) := (g_j,[t_ig_ig_j^{-1}]).$$
This map is well-defined, smooth, and $G$-equivariant where $G$ acts trivially on $BG$ and via $k\cdot g=gk^{-1}$ on itself.  Moreover, $\varphi_j$ is invertible with smooth inverse: $$\varphi_j^{-1}(k,[t_ig_i]) = (t_ig_ik)$$ where we set $g_j=e$.  Thus, $\{(U_j,\varphi_j)\}$ is a local trivialisation of $EG\to BG$.  It follows that $EG\to BG$ is a locally trivial diffeological principal $G$-bundle.

Finally, for each $(t_ig_i)\in EG,$
$$\sum_{j \in \NN} s_j(t_ig_i) = \sum_{j \in \NN} t_j = 1,$$ and each $s_j$ is $G$-invariant and so descends to a smooth map $\zeta_j\colon BG\to[0,1]$ with $\zeta_j^{-1}(0,1]=U_j$.  The collection $\{\zeta_j\}$ is a pointwise-finite smooth partition of unity on $BG$, and so we have shown that $\pi\colon EG\to BG$ is weakly D-numerable.
\end{proof}

%% %% %%
\subsection{Classifying Bundles}\labell{ss:classifying}
%% %% %%

The main result concerning a classifying space of a diffeological group $G$ is that it classifies principal $G$-bundles, up to isomorphism.  To establish this, we follow the topological presentation by tom Dieck (see \cite[Sections 14.3, 14.4]{tD}).  While many of the proofs only require slight modifications to ensure smoothness, some such as the proof to Proposition~\ref{p:homotopy} requires the development of some diffeological theory related to the D-topology and homotopy, which appears in Lemmas~\ref{l:open cover}, \ref{l:smoothly T4}, and \ref{l:homotopy}.  To make the classification result precise, we introduce the following notation.

\begin{definition}[$\mathcal{B}_G(\cdot)$ and $\lbrack\cdot,BG\rbrack$]\labell{d:B}
Let $G$ be a diffeological group, and let $X$ be diffeological space.  Denote by $\mathcal{B}_G(X)$ the set of isomorphism classes of D-numerable principal $G$-bundles over $X$, and denote by $[X,BG]$ the set of smooth homotopy classes of smooth maps $X\to BG$.  Given another diffeological space $Y$ and a smooth map $\varphi\colon X\to Y$, define $\mathcal{B}(\varphi)$ to be the pullback $\varphi^*\colon \mathcal{B}_G(Y)\to\mathcal{B}_G(X)$, and $[\varphi,BG]$ to be the pullback $\varphi^*\colon [Y,BG]\to[X,BG]$.
\end{definition}

The main theorem of this section is the following.

\begin{theorem}[$\mathcal{B}_G(\cdot)$ and $\lbrack\cdot,BG\rbrack$ are Naturally Isomorphic]\labell{t:BG}
Let $\mathbf{Diffeol}_{\operatorname{HSP}}$ be the full subcategory of $\mathbf{Diffeol}$ consisting of Hausdorff, second-countable, smoothly paracompact diffeological spaces, and let $G$ be a diffeological group.  Then, $\mathcal{B}_G(\cdot)$ and $[\cdot,BG]$ are naturally isomorphic functors from $\mathbf{Diffeol}_{\operatorname{HSP}}$ to $\mathbf{Set}$.
\end{theorem}

To prove this, we begin by constructing a correspondence $[X,BG]\to\mathcal{B}_G(X)$.

\begin{proposition}[Pullback Bundles]\labell{p:EG classifying}
Let $X$ be a diffeological space, and let $G$ be a diffeological group.
For any smooth map $F\colon X \to BG$ the pullback bundle $F^*EG$ is weakly D-numerable.  Additionally, if $X$ is Hausdorff and smoothly paracompact, then $F^*EG$ is D-numerable.
\end{proposition}

\begin{proof}
Let $F\colon X\to BG$ be a smooth map.  Let the collection $\{\zeta_j\colon BG\to[0,1]\}_{j\in\NN}$ be the partition of unity in the proof of Proposition~\ref{p:EG principal}.  For each $j\in\NN$, define $\xi_j\colon X\to[0,1]$ by $\xi_j:=\zeta_j\circ F.$ Then $\{\xi_j\}$ is a pointwise-finite smooth partition of unity, and we have an open cover of $X$ given by open sets $W_j:=\xi_j^{-1}(0,1]$.  We now show that $F^*EG$ is trivial over each $W_j$.

Recall $$F^*EG=\{(x,(t_ig_i))\in X\times EG\mid F(x)=[t_ig_i]\},$$ where for all $k\in G$ we have the smooth action $k\cdot(x,(t_ig_i))=(x,(t_ig_ik^{-1}))$. Define $\widetilde{F}\colon F^*EG\to EG$ to be the second projection map.  Define $\psi_j\colon \pi^{-1}(W_j)\to G\times W_j$ by $$\psi_j(x,(t_ig_i)):=(g_j,x).$$ Then $\psi_j$ is well-defined, smooth, $G$-equivariant, and has a smooth inverse.  It follows that $F^*EG$ is weakly D-numerable.

If $X$ is Hausdorff and smoothly paracompact, we can choose an appropriate open refinement of $\{W_j\}$ and smooth partition of unity subordinate to it (see, for example, \cite[Lemma 41.6]{munkres}).  This completes the proof.
\end{proof}

To ensure that the correspondence $[X,BG]\to\mathcal{B}_G(X)$ is well-defined, we must show that smoothly homotopic maps yield isomorphic bundles.  This is Proposition~\ref{p:homotopy}.  To prove this, we need to add another topological constraint, second-countability, and establish a series of lemmas.  Also, we need to consider the D-topology of a product of diffeological spaces.  Generally, the D-topology induced by the product diffeology contains the product topology.  However, we have the following result of Christensen, Sinnamon, and Wu in \cite[Lemma 4.1]{CSW}.

\begin{lemma}\labell{l:prod top}
Let $X$ and $Y$ be diffeological spaces such that the D-topology of $Y$ is locally compact.  Then the D-topology of the product $X\times Y$ is equal to the product topology.
\end{lemma}

\begin{lemma}\labell{l:open cover}
Let $X$ be a Hausdorff, paracompact, second-countable diffeological space, and let $\pi\colon E\to X\times\RR$ be a locally trivial diffeological fibre bundle.   Then for any $a<b$, there exists a countable locally finite open cover $\mathcal{U}$ of $X$ such that $\pi$ is trivial over $U\times[a,b]$ for any $U\in\mathcal{U}$.  In the case of a principal bundle, the trivialisation can be chosen to be equivariant.
\end{lemma}

\begin{proof}
Since $\RR$ is locally compact, the D-topology on $X\times\RR$ is equal to the product topology by Lemma~\ref{l:prod top}.  Fix a local  trivialisation (equivariant in the case of a principal bundle) $\mathcal{V}$ of $\pi$, and fix $x\in X$.  Since $[a,b]$ is compact there exist $k>0$, for each $i=1,\dots,k$ an open interval $(s_i,t_i)$ such that $\bigcup_i(s_i,t_i)\supseteq[a,b]$, and for each $i$ an open neighbourhood $U_i$ of $x$ such that $U_i\times(s_i,t_i)$ is contained in an element of $\mathcal{V}$.

Let $U=\bigcap_iU_i$.  By \cite[Note 1 of Article 8.16 and Lemma 1 of Article 8.19]{iglesias}, $\pi$ is trivial over $U\times[a,b]$.  Now, take such an open neighbourhood $U$ for each $x\in X$; this is an open cover of $X$.  We now take an appropriate refinement of this cover to obtain $\mathcal{U}$. 
\end{proof}

It is standard that given a normal topological space, one can find a continuous function that separates disjoint closed sets.  We need a smooth version of this fact.  

\begin{lemma}\labell{l:smoothly T4}
Let $X$ be a smoothly paracompact diffeological space, and let $A$ and $B$ be disjoint closed subsets of $X$.  Then there exists a smooth function $f\colon X\to[0,1]$ such that $f|_A\equiv 0$ and $f|_B\equiv 1$.
\end{lemma}

\begin{proof}
Consider the open cover $\{X\smallsetminus A,X\smallsetminus B\}$.  It admits a smooth partition of unity $\{\zeta_{X\smallsetminus A},\zeta_{X\smallsetminus B}\}$.  Let $f=\zeta_{X\smallsetminus A}$.\footnote{Thanks to Dan Christensen for this proof, which is much simpler than the original.}
\end{proof}

\begin{lemma}\labell{l:homotopy}
Let $X$ be a Hausdorff, second-countable, smoothly paracompact diffeological space, and let $\pi\colon E\to X\times\RR$ be a locally trivial diffeological fibre bundle.   Then $E|_{X\times\{0\}}$ is bundle-diffeomorphic to $E|_{X\times\{1\}}$.  In the case of a principal bundle, the bundle-diffeomorphism can be chosen to be equivariant.
\end{lemma}

\begin{proof}
By Lemma~\ref{l:open cover}, there exists a locally finite open cover $\{U_i\}_{i\in\NN}$ of $X$ such that $\pi$ is trivial over $U_i\times[0,1]$ for each $i$.  Let $\{(U_i\times[0,1],\psi_i)\}$ be the corresponding local trivialisation (equivariant in the case of a principal bundle) for $\pi$ restricted over $X\times[0,1]$.  Since $X$ is Hausdorff and paracompact, there is a locally finite open refinement $\{V_i\}_{i\in\NN}$ of $\{U_i\}$ such that $\overline{V_i}\subseteq U_i$ for each $i$.  By Lemma~\ref{l:smoothly T4} there is a family of smooth maps $\{b_i\colon X\to[0,1]\}_{i\in\NN}$ such that $b_i|_{V_i}=1$ and $\supp(b_i)\subseteq U_i$ for each $i$.

Fix $i$, and denote by $F$ the fibre of $\pi$.  Let $r_i\colon X\times[0,1]\to X\times[0,1]$ be the smooth map sending $(x,t)$ to $(x,t+(1-t)b_i(x))$ and let $R_i\colon E\to E$ be the map equal to the identity over the complement of $U_i\times[0,1]$, and such that for all $(x,t,a)\in U_i\times[0,1]\times F$, $$\psi_i\circ R_i\circ\psi_i^{-1}(x,t,a)=(x,t+(1-t)b_i(x),a).$$  The pair $(r_i,R_i)$ form a smooth bundle map over $X\times[0,1]$, whose restriction to $E|_{X\times\{1\}}$ is the identity map.  Moreover, the restriction of $R_i$ to $E|_{X\times\{0\}}$ is a diffeomorphism onto its image.  In the case of a principal bundle, this is equivariant.

Let $r$ be the composition of the maps $r_i$, taken in order: $r=\dots\circ r_2\circ r_1$.  This is well-defined since $r_i(x)=x$ for all but finitely many $i$.  Similarly, define $R$ to be the composition of all $R_i$.  The pair $(r,R)$ is a smooth bundle map, whose restriction to $E|_{X\times\{1\}}$ is the identity map.  Moreover, the restriction of $R$ to $E|_{X\times\{0\}}$ is a diffeomorphism onto $E|_{X\times\{1\}}$.  Again, in the case of a principal bundle, this is also equivariant.
\end{proof}

\begin{proposition}[Homotopic Maps and Isomorphic Bundles]\labell{p:homotopy}
Let $X$ and $Y$ be diffeological spaces, and assume that $X$ is Hausdorff, second-countable, and smoothly paracompact.  Let $f_i\colon X\to Y$ ($i=0,1$) be smooth maps with a smooth homotopy $H\colon X\times\RR\to Y$ between them, and let $\pi\colon E\to Y$ be a diffeological fibre bundle.  If $H^*E\to X$ is locally trivial, then the pullback bundles $f_i^*E\to X$ ($i=0,1$) are bundle-diffeomorphic.  In the case of a principal bundle, the bundle-diffeomorphism can be chosen to be equivariant.
\end{proposition}

\begin{proof}
Assume that $H^*E\to X$ is locally trivial.  Then there is a bundle-diffeomorphism between $H^*E|_{X\times\{0\}}$ and $H^*E|_{X\times\{1\}}$ by Lemma~\ref{l:homotopy}.  But these two restricted bundles are exactly bundle-diffeomorphic to $f_0^*E$ and $f_1^*E$, respectively.
\end{proof}

\begin{remark}\labell{r:homotopy}
Iglesias-Zemmour proves the above proposition for diffeological fibre bundles admitting a diffeological connection; see Definition~\ref{d:iz-connection}, \cite[Article 8.32 and Article 8.34]{iglesias}.
\end{remark}

We now have a well-defined correspondence $[X,BG]\to\mathcal{B}_G(X)$.  To construct an inverse correspondence, we must first define a ``classifying map'' $X\to BG$ for a principal $G$-bundle $E\to X$.  In order to show that the classifying map is smooth, we require that $E$ be weakly D-numerable with a Hausdorff and smoothly paracompact base (in which case, $E$ is D-numerable).

\begin{proposition}[Classifying Maps]\labell{p:classifying maps}
If $\pi\colon E\to X$ is a weakly D-numerable principal $G$-bundle in which $X$ is Hausdorff and smoothly paracompact, then there is a smooth $G$-equivariant map $\widetilde{F}\colon E\to EG$ which descends to a smooth map $F\colon X\to BG$, called a \textbf{classifying map} of $\pi$, such that $E$ is isomorphic as a principal $G$-bundle to $F^*EG$.  
\end{proposition}

\begin{proof}
Let $\pi\colon E\to X$ be a weakly D-numerable principal $G$-bundle in which $X$ is Hausdorff and smoothly paracompact (and hence $\pi$ is in fact D-numerable by Remark~\ref{r:d-numerable}), and fix a locally finite local trivialisation $\{(W_j,\psi_j)\}_{j\in\NN}$ and a smooth partition of unity $\{\xi_j\}_{j\in\NN}$ subordinate to $\{W_j\}$.  Define a map $\widetilde{F}\colon E\to EG$ by $$\widetilde{F}(y):=\big((\xi_i\circ\pi(y))(\pr_2\circ\psi_i(y))\big).$$  Since $\xi_j$ has support contained in $W_j$, $\widetilde{F}$ is well-defined.

To show that $\widetilde{F}$ is smooth, fix a plot $p\colon U\to E$.  It is enough to show that $\widetilde{F}\circ p$ locally lifts to a plot of $S_G$.  Fix $u\in U$.  Since $\{W_j\}$ is a locally finite open cover, there is an open neighbourhood $B$ of $p(u)$ such that $B$ intersects only finitely many of the open sets $\pi^{-1}(W_j)$; without loss of generality, assume that these are $W_1,\dots,W_k$.  Next, of these, only $l\leq k$ contain $p(u)$.  Again, without loss of generality, assume that these are $W_1,\dots,W_l$.  Since $\{\xi_j\}$ is subordinate to $\{W_j\}$, we have that the closed set $\bigcup_{j=1}^{k-l}\supp(\xi_{l+j})$ is disjoint from $p(u)$, which itself is closed since $X$ is Hausdorff.  Furthermore, since $X$ is also paracompact, it is normal, and so shrink $B$ so that it only intersects $W_{l+1},\dots,W_k$ outside of each $\supp(\xi_j)$ ($j=l+1,\dots,k$).   

Let $p_B:=p|_{p^{-1}(B)}$ and let $\rho\colon S_G\to EG$ be the quotient map.  Then, $\widetilde{F}\circ p_B=\rho\circ\sigma_B$ where $\sigma_B\colon p^{-1}(B)\to S_G$ is defined to be the smooth map $\sigma_B(u)=(t_i(u),g_i(u))$ where $g_i(u)=e$ if $i>l$, $g_i(u)=\pr_2(\psi_i(p_B(u)))$ if $i=1,\dots,l$, and $t_i(u)=\xi_i(\pi(p_B(u)))$ for each $i$.  It follows that $\widetilde{F}$ is a smooth map into $EG$.  It is also $G$-equivariant, and so descends to a smooth map $F\colon X\to BG$.

To show that $E$ and $F^*EG$ are isomorphic, define a map $\Phi\colon E\to F^*EG$ by $\Phi(y):=(\pi(y),\widetilde{F}(y))$.  Then, $\Phi$ is a well-defined smooth bijection.  Fixing $x\in X$, let $B$ be an open neighbourhood of $x$ intersecting only finitely many $W_j$.  Then $\Phi^{-1}|_{\pi^{-1}(B)}(x,(t_ig_i))=\psi_j^{-1}(x,g_j)$ for any $j$ such that $x\in W_j$, which is smooth, and hence $\Phi$ is a diffeomorphism.  $G$-equivariance of $\Phi$ follows from the $G$-equivariance of $\widetilde{F}$.
\end{proof}

\begin{remark}\labell{r:classifying maps}
If $\pi\colon E\to X$ is a principal $G$-bundle that has a smooth classifying map $F\colon X\to BG$, then it follows that $F^*EG$ is weakly D-numerable, and so by definition of a classifying map, $\pi\colon E\to X$ is also weakly D-numerable.
\end{remark}

To show that our inverse correspondence $\mathcal{B}_G(X)\to[X,BG]$ is well-defined, we need to show that isomorphic bundles yield smoothly homotopic classifying maps.  This is a consequence of the following proposition.

\begin{proposition}[Homotopy Equivalence of Maps to EG]\labell{p:uniqueness}
Let $G$ be a diffeological group, $E\to X$ a principal $G$-bundle, and $f,h\colon E\to EG$ two $G$-equivariant smooth maps.  Then, $f$ and $h$ are smoothly $G$-equivariantly homotopic, and so descend to smoothly homotopic maps $X\to BG$.
\end{proposition}

\begin{proof}
Denote the images of $f$ and $h$ by $$f(y)=\big(s_1(y)f_1(y),s_2(y)f_2(y),\dots\big) \text{ and } h(y)=\big(t_1(y)h_1(y),t_2(y)h_2(y),\dots\big)$$ where $f_i(y),h_i(y)\in G$ and $s_i(y),t_i(y)\in[0,1]$.  We will construct an homotopy $H$ such that $H(y,0)=f(y)$ and $H(y,1)=h(y)$ for all $y\in E$.  The construction will be a concatenation of homotopies that we construct now.  Throughout these, we use a smooth function $b\colon\RR\to[0,1]$ such that $b(\tau)=0$ for all $\tau\leq \eps$, $b(\tau)=1$ for all $\tau\geq 1-\eps$, and $\frac{db}{d\tau}\geq 0$, for some fixed $\eps\in(0,1/2)$.

Let $F_1\colon E\times[0,1]\to EG$ be defined by 
\begin{align*}
F_1(y,\tau):=&~\big(s_1(y)f_1(y),b(\tau)s_2(y)f_2(y),(1-b(\tau))s_2(y)f_2(y),\\
&~b(\tau)s_3(y)f_3(y),(1-b(\tau))s_3(y)f_3(y),\dots\big).
\end{align*}
Then $F_1$ is smooth, and satisfies
\begin{gather*}
F_1(y,0)=\big(s_1(y)f_1(y),0,s_2(y)f_2(y),0,\dots\big) \text{ and }\\
F_1(y,1)=\big(s_1(y)f_1(y),s_2(y)f_2(y),0,s_3(y)f_3(y),0,\dots\big).
\end{gather*}
Similarly, for each $n>1$, let $F_n\colon E\times[0,1]\to EG$ be defined by 
\begin{align*}
F_n(y,\tau):=&~\big(s_1(y)f_1(y),\dots,s_n(y)f_n(y),b(\tau)s_{n+1}(y)f_{n+1}(y),\\
&~(1-b(\tau))s_{n+1}(y)f_{n+1}(y),b(\tau)s_{n+2}(y)f_{n+2}(y),\dots\big).
\end{align*}
Each $F_n$ is smooth, and $F_n(y,0)=F_{n-1}(y,1)$ for all $n>1$.

Define $F\colon E\times[0,\infty)\to EG$ by $F(y,\tau)=F_{\lfloor\tau+1\rfloor}(y,\tau-\lfloor\tau\rfloor),$ where $\lfloor\cdot\rfloor$ is the floor function.  Since each $F_n$ is constant near $\tau=0$ and $\tau=1$ for fixed $y$, it follows that $F$ is smooth.  Let $c\colon[0,1)\to[0,\infty)$ be an increasing diffeomorphism ($x\mapsto\tan(\frac{\pi}{2}x)$, say), and define $F'\colon E\times[0,1)\to EG$ as the smooth composition $F\circ(\id_E\times c)$.  Define $F'(y,1):=f(y)$.  This yields a smooth extension $F'\colon E\times[0,1]\to EG$.

Similarly, for each $n>0$, define smooth maps $S_n\colon E\times[0,1]\to EG$ by 
\begin{align*}
S_n(y,\tau):=&~\big(t_1(y)h_1(y),\dots,t_{n-1}(y)h_{n-1}(y),b(\tau)t_n(y)h_n(y),\\
&~(1-b(\tau))t_n(y)h_n(y),b(\tau)t_{n+1}(y)h_{n+1}(y),\dots\big).
\end{align*}
We have $S_n(y,0)=S_{n-1}(y,1)$ for each $n>1$.  Defining $S$ and $S'$ similar to $F$ and $F'$, by setting $S'(y,1)=h(y)$ we obtain a smooth extension $S'\colon E\times[0,1]\to EG$.

Finally, define $T\colon E\times[0,1]\to EG$ by 
\begin{align*}
T(y,\tau):=&~\big((1-b(\tau))s_1(y)f_1(y),b(\tau)t_1(y)h_1(y),\\
&~(1-b(\tau))s_2(y)f_2(y),b(\tau)t_2(y)h_2(y),\dots\big).
\end{align*}
Then $T$ is smooth, $T(y,0)=F'(y,0)$, and $T(y,1)=S'(y,0)$.

We now define the smooth homotopy $H\colon E\times[0,1]\to EG$ between $f$ and $h$ by
$$H(y,\tau):=\begin{cases}
F'(y,1-b(3\tau)) & \text{for $\tau\in[0,\frac{1}{3}]$,}\\
T(y,b(3\tau-1)) & \text{for $\tau\in[\frac{1}{3},\frac{2}{3}]$,}\\
S'(y,b(3\tau-2)) & \text{for $\tau\in[\frac{2}{3},1]$.}\\
\end{cases}$$
The $G$-equivariance of $H$ is clear.  This completes the proof.
\end{proof}

We are now ready to prove the main result of this section.

\begin{proof}[Proof of Theorem~\ref{t:BG}]
The fact that $\mathcal{B}_G(\cdot)$ and $[\cdot,BG]$ are functors is clear. Fix a diffeological space from $\mathbf{Diffeol}_{\operatorname{HSP}}$.  It follows from Proposition~\ref{p:EG classifying} and Proposition~\ref{p:homotopy} that there is a map $\alpha$ from the set $[X,BG]$ to $\mathcal{B}_G(X)$.   It follows from Proposition~\ref{p:classifying maps} and Proposition~\ref{p:uniqueness} that $\alpha$ has an inverse, and hence $\mathcal{B}_G(X)$ is in bijection with $[X,BG]$.  Finally, this bijection is natural: given two diffeological spaces $X$ and $Y$ from $\mathbf{Diffeol}_{\operatorname{HSP}}$ and a smooth map $\varphi\colon X\to Y$, we have that $\mathcal{B}(\varphi)$ and $[\varphi,BG]$ commute with these bijections.
\end{proof}

%%%%%%%%%%%%%%%%%%%%%%%%%%%%%%%%%%%%%%%%%%%%%%%%%%%%%%
\section{Diffeological Connections and Connection $1$-Forms}\labell{s:connection}
%%%%%%%%%%%%%%%%%%%%%%%%%%%%%%%%%%%%%%%%%%%%%%%%%%%%%%

In \cite[Article 8.32]{iglesias} Iglesias-Zemmour gives a definition of a connection on a principal $G$-bundle in terms of paths on the total space, generalising the classical notion for principal bundles with structure group a finite-dimensional Lie group.  From this definition one obtains the usual properties that one expects from a connection (see Remark~\ref{r:iz-connection}).  The purpose of this section is to prove Theorem~\ref{t:iz-connection}: that any principal $G$-bundle satisfying mild conditions admits one of these connections provided $G$ is a regular diffeological Lie group.  We do this by way of constructing a connection $1$-form on the $G$-bundle $EG\to BG$, and showing that this induces a connection in the sense of Iglesias-Zemmour.  Since connections pull back, we obtain our result.

\begin{definition}[Diffeological Connections]\labell{d:iz-connection}
Let $G$ be a diffeological group, and let $\pi\colon E\to X$ be a principal $G$-bundle.  Denote by $\pathloc(E)$ the diffeological space of \textbf{local paths} $$\pathloc(E):=\{\gamma\in\CIN((a,b),E)\mid (a,b)\subseteq\RR\}$$ equipped with the subset diffeology induced by the standard functional diffeology \emph{on a diffeology} (see \cite[Article 1.63]{iglesias}).  We recall what this standard functional diffeology is: a parametrisation $p:U\to\mathcal{D}$ is a plot if and only if for each $u\in U$ and $v\in\dom(p(u))$ there exist open neighbourhoods $U'$ of $u$ and $V'$ of $v$ such that for each $u'\in U'$,
\begin{enumerate}
\item $V'\subseteq\dom(p(u'))$, and
\item\labell{i:diff of plots} the map $\Psi\colon U'\times V'\to X\colon(u',v')\mapsto p(u')(v')$ is in $\mathcal{D}$.
\end{enumerate}

Denote by $\tbpath(E)$ the \textbf{tautological bundle of local paths}, equipped with the subset diffeology induced by $\pathloc(E)\times\RR$: $$\tbpath(E):=\{(\gamma,t)\in\pathloc(E)\times\RR\mid t\in\dom(\gamma)\}.$$ A \textbf{diffeological connection} is a smooth map $\theta\colon\!\tbpath(E)\to\pathloc(E)$ satisfying the following properties for any $(\gamma,t_0)\in\tbpath(E)$:
\begin{enumerate}
\item the domain of $\gamma$ equals the domain of $\theta(\gamma,t_0)$,
\item $\pi\circ\gamma=\pi\circ\theta(\gamma,t_0)$,
\item $\theta(\gamma,t_0)(t_0)=\gamma(t_0)$,
\item $\theta(g\cdot\gamma,t_0)=g\cdot\theta(\gamma,t_0)$ for all $g\in G$,
\item $\theta(\gamma\circ f,s)=\theta(\gamma,f(s))\circ f$ for any smooth map $f$ from an open subset of $\RR$ into $\dom(\gamma)$,
\item $\theta(\theta(\gamma,t_0),t_0)=\theta(\gamma,t_0)$. 
\end{enumerate}
\end{definition}

\begin{remark}\labell{r:iz-connection}
Diffeological connections satisfy many of the usual properties that classical connections on a principal $G$-bundle (where $G$ is a finite-dimensional Lie group) enjoy; in particular, they admit unique horizontal lifts of paths into the base of a principal bundle \cite[Article 8.32]{iglesias}, and they pull back by smooth maps \cite[Article 8.33]{iglesias}.
\end{remark}

We now state the main purpose of this section.

\begin{theorem}[Diffeological Connections on Principal $G$-Bundles]\labell{t:iz-connection}
Let $G$ be a regular diffeological Lie group, and let $X$ be a Hausdorff smoothly paracompact diffeological space.  Then any weakly D-numerable principal $G$-bundle $E\to X$ admits a diffeological connection.  Consequently, for this diffeological connection, any smooth curve into $X$ has a unique horizontal lift to $E$.
\end{theorem}

To prove this, we begin by constructing a connection 1-form on $EG\to BG$.

\begin{definition}[Connection 1-Form]\labell{d:conn 1-form}
Let $G$ be a diffeological group, and let $E\to X$ be a principal $G$-bundle.  A \textbf{connection 1-form} on $E$ is a $G$-equivariant smooth fibrewise linear map $\omega\colon TE\to T_eG$ (with respect to the adjoint action on $T_{e}G$) such that for any $y\in E$ and $\xi\in T_eG$, we have $\omega(\xi_E|_y)=\xi$ where $$\xi_E|_y:=\frac{d}{dt}\Big|_{t=0}(g(t)\cdot y)$$ in which $g(t)$ is a smooth curve in $G$ such that $g(0)=e$ and $\dot{g}(0)=\xi$.
\end{definition}

To make clear the connection between ordinary diffeological differential $1$-forms and smooth fibrewise linear maps as defined above, we present the following definition and proposition.

\begin{definition}\labell{d:V-valued forms}
Denote by $\DVB$ the category of diffeological vector pseudo-bundles with smooth fibrewise linear maps between them.  For a diffeological space $(X,\mathcal{D})$ denote by $\operatorname{Plots}(\mathcal{D})$ the category with objects plots in $\mathcal{D}$ and arrows commutative triangles 
$$\xymatrix{
U \ar[rr]^{f} \ar[dr]_{p} & & V \ar[dl]^{q} \\
 & X & \\
}$$
in which $f$ is smooth.  Let $V$ be a diffeological vector space.  Denote $\Omega^1(X;V)$ to be the \textbf{$V$-valued differential $1$-forms on $X$}, defined to be the limit $$\underset{\operatorname{Plots}(\mathcal{D})}{\lim}\Hom_{\DVB}(T\circ F(\cdot),X\times V\to X).$$  Here, $X\times V\to X$ is the trivial $V$-bundle over $X$, $T$ is the tangent functor, and $F$ the forgetful functor sending plots to their domains and commutative triangles to the corresponding maps between Euclidean open sets.
\end{definition}

\begin{proposition}\labell{p:V-valued forms}
Let $(X,\mathcal{D})$ be a diffeological space and $V$ a diffeological vector space.  Then there is a natural identification between smooth fibrewise $\RR$-linear maps $TX\to V$ and $V$-valued $1$-forms $\Omega^1(X;V)$.\footnote{Proof suggested by Daniel Christensen via private communication.}
\end{proposition}

\begin{proof}
The set of smooth fibrewise $\RR$-linear maps $TX\to V$ is exactly $$\Hom_{\DVB}(TX,X\times V\to X).$$  But $TX$ is the colimit $\colim(T\circ F)$ in $\DVB$ over the category $\operatorname{Plots}(\mathcal{D})$ where $T$ and $F$ are as in Definition~\ref{d:V-valued forms}; see \cite[Theorem 4.17]{CW}.  But then,
$$\Hom_{\DVB}(TX,X\times V\to X)=\underset{\operatorname{Plots}(\mathcal{D})}{\lim}\Hom_{\DVB}(T\circ F(\cdot),X\times V\to X);$$
see, for example, \cite[Corollary 5.29]{awodey}.
\end{proof}

\begin{remark}\labell{r:V-valued forms}
Ordinary differential $1$-forms on a diffeological space $X$ are, by definition, equal to $$\underset{\operatorname{Plots}(\mathcal{D})}{\lim}\Hom_{\DVB}(T\circ F(\cdot),X\times \RR\to X)$$
\end{remark}

Returning to connection $1$-forms, the most basic connection $1$-form is the Maurer-Cartan form on a diffeological group.

\begin{lemma}[Maurer-Cartan Form]\labell{l:maurer-cartan}
Let $G$ be a diffeological group.  Then $G$ has a \textbf{Maurer-Cartan form} $\alpha$; that is, a smooth fibrewise-linear map $\alpha\colon TG\to T_eG$ sending $v\in T_gG$ to $(L_{g^{-1}})_*v$.  It is $G$-equivariant with respect to the adjoint action on $T_eG$ and the left action on $TG$ defined by $h\cdot v:= (R_{h^{-1}})_*v$.
\end{lemma}

\begin{proof}
This is immediate from the trivialisation of $TG$ given in Theorem~\ref{t:group}.
\end{proof}

Since $EG$ is constructed out of an infinite product of the group $G$, we can take the infinite sum of the Maurer-Cartan form, with coefficients $t_j$; since only finitely many of the $t_j$ are non-zero, the sum converges.  The goal is to show that the result is smooth.

\begin{theorem}[Connection on $EG$]\labell{t:universal connection}
Let $G$ be a diffeological group.  Then the principal $G$-bundle $EG\to BG$ admits a connection 1-form $\omega$.
\end{theorem}

\begin{proof}[Proof of Theorem~\ref{t:universal connection}]
Let $\alpha$ be the Maurer-Cartan form on $G$.  Define $\widetilde{\omega}$ to be the smooth map $\widetilde{\omega}\colon TS_G\to T_eG$ given by
$$\widetilde{\omega}|_{(g_i,t_i)} = \sum_{i \in \NN} t_i \pr_{g_i}^*\alpha$$
where $\pr_{g_i}\colon S_G\to G$ is the projection map onto the $i^\text{th}$ copy of $G$.  Then $\widetilde{\omega}$ is a well-defined connection 1-form on the bundle $S_G\to S_G/G=:EG$.

We need to show that $\widetilde{\omega}$ descends to a form $\omega$ on $EG$.  To do this, we use the following fact (see \cite[Article 6.38]{iglesias}, noting that the proof goes through with our definition of $1$-form, independent of the fact that the forms in the proof are real-valued): if $\rho\colon S_G\to EG$ is the quotient map, then a form $\mu$ on $S_G$ is equal to $\rho^*\nu$ for some form $\nu$ on $EG$ if and only if for any two plots $p_1,p_2:U\to S_G$ satisfying $\rho\circ p_1=\rho\circ p_2$, we have $p_1^*\mu=p_2^*\mu$.  Fix two such plots $p_1$ and $p_2$, and fix $u\in U$ and $v\in T_uU$.  It is enough for us to show
\begin{equation}\labell{e:universal connection}
v\hook (p_1^*\widetilde{\omega}-p_2^*\widetilde{\omega})=0.
\end{equation}

There are only finitely many $j\in\NN$ such that $\pr_{t_j}(p_1(u))\neq 0$ where $\pr_{t_j}$ is the $j^{\text{th}}$ projection from $S_G$ onto $[0,1]$.  Hence there are only finitely many open sets $$\widetilde{V}_j=\{(g_i,t_i)\mid t_j\neq 0\}\subseteq S_G$$ containing $p_1(u)$.  Thus, the intersection $\widetilde{V}$ of all such $\widetilde{V}_j$ is open, and its pre-image $W:=p_1^{-1}(\widetilde{V})\subseteq U$ is open.  Moreover, it follows from $\rho\circ p_1=\rho\circ p_2$ that $W=p_2^{-1}(\widetilde{V})$.  

Let $c\colon\RR\to U$ be a curve such that $c(0)=u$ and $\dot{c}(0)=v$.  Then for all $\tau$ in $c^{-1}(W)$ we have $\pr_{g_j}\circ p_1\circ c(\tau)=\pr_{g_j}\circ p_2\circ c(\tau)$ for all $j$ such that $t_j\neq 0$.  It follows that $$\alpha((\pr_{g_j})_*(p_1)_*v)=\alpha((\pr_{g_j})_*(p_2)_*v).$$  Equation~\eqref{e:universal connection} follows.
\end{proof}

We have the following corollary:

\begin{corollary}\labell{c:universal connection}
Let $G$ be a diffeological group.  Then any principal $G$-bundle $E\to X$ admitting a smooth classifying map has a connection 1-form.  In particular, if $E\to X$ is weakly D-numerable over a Hausdorff smoothly paracompact diffeological space $X$, then it admits a connection 1-form.
\end{corollary}

\begin{proof}
This follows from the existence of a $G$-equivariant map $E\to EG$ (Proposition~\ref{p:classifying maps}).
\end{proof}

In order to continue, we need the following basic lemma about differentiating plots.

\begin{lemma}\labell{l:diff of plots}
Let $(X,\mathcal{D})$ be a diffeological space, and let $T\colon\mathcal{D}\to T\mathcal{D}$ be the map sending a plot $p$ to its induced plot $Tp$.  Then $T$ is a smooth map.
\end{lemma}

\begin{proof}
Fix a plot $p\colon U\to\mathcal{D}$ in the standard functional diffeology.  Using the same notation as in Definition~\ref{d:iz-connection}, by the definition of the diffeology on $TX$ and Remark~\ref{r:tangent bdle} the map $T\Psi\colon TU'\times TV'\to TX$ is a plot of $TX$.  Restricting the first coordinate of $T\Psi$ to the zero-section of $TU'$ we get the smooth map $U'\times TV'\to TX\colon (u',w)\mapsto T\Psi|_{(u',v')}(0,w)$ where $v'$ is the foot-point of $w$.  Since this is smooth, by the definition of the functional diffeology we have that the map $U'\to\CIN(TV',TX)$ sending $u'$ to $Tp(u')$ is a plot of $T\mathcal{D}$.
\end{proof}

Now the main obstacle in proving Theorem~\ref{t:iz-connection} is to show that a connection 1-form on a principal $G$-bundle yields a diffeological connection; in particular, that one obtains a smooth map from the tautological bundle to the local paths of the total space.  This is the content of the following proposition (\emph{cf.} \cite[Section 2.1]{Ma2013}).

\begin{proposition}[Connection 1-Forms Induce Diffeological Connections]\labell{p:iz-connection}
Let $G$ be a regular diffeological Lie group, and let $\pi\colon E\to X$ be a principal $G$-bundle.  Then a connection 1-form $\omega$ on $\pi$ induces a diffeological connection on $\pi$.
\end{proposition}

\begin{proof}
Fix a smooth curve $\gamma\colon(a,b)\to E$ and a point $t_0\in(a,b)$.  Our first goal is to obtain a smooth curve $g\colon(a,b)\to G$ such that $g(t_0)=e$ and the smooth curve $t\mapsto g(t)\cdot\gamma(t)$ satisfies $$\omega\left(\frac{d(g\cdot \gamma)(t)}{dt}\right)=0.$$  We will denote $g(t)\cdot\gamma(t)$ by $\theta(\gamma,t_0)(t)$.  By the chain rule (see Remark~\ref{r:tangent bdle})  the derivative of $\theta(\gamma,t_0)$ is $g\cdot\dot{\gamma}+\dot{g}\cdot\gamma$.  Applying $\omega$ to this, we obtain the differential equation $$\omega(\dot{\gamma})(t)=-\Ad_{g^{-1}}\dot{g}|_{\gamma(t)}.$$ 

Choose $\eps>0$ so that $a<t_0-\eps$ and $t_0+\eps<b$.  After composing with an appropriate translation and dilation, it follows from the regularity of $G$ that there is a smooth solution $[t_0-\eps,t_0+\eps]\to G$; in fact, uniqueness of solutions implies that applying this procedure to each $t_0\in(a,b)$ will yield a smooth curve $g\colon (a,b)\to G$ as required.  We thus have proved the existence of $\theta\colon\!\tbpath(E)\to\pathloc(E)$.  

To show that $\theta$ is smooth, note that the map $\gamma\mapsto\dot{\gamma}$ is smooth by Lemma~\ref{l:diff of plots}, as well as $\omega$ and the exponential map in the definition of the regularity of $G$.  Finally, the translations and dilations are smooth, and since $\theta$ is a composition of all of these things, smoothness follows.

It is an easy exercise to check that $\theta$ satisfies the six conditions in Definition~\ref{d:iz-connection}.  This completes the proof.
\end{proof}

\begin{proof}[Proof of Theorem~\ref{t:iz-connection}]
By Corollary~\ref{c:universal connection} the principal $G$-bundle $\pi\colon E\to X$ has a connection 1-form.  By Proposition~\ref{p:iz-connection} this induces a diffeological connection on $\pi$.  In particular, we obtain horizontal lifts of smooth curves into $X$ (see Remark~\ref{r:iz-connection}).
\end{proof}

\begin{remark}\labell{r:homotopy and conn}
We end this section with the following note.  Let $G$ be a \emph{regular} diffeological \emph{Lie} group.  By Proposition~\ref{p:iz-connection} we have a diffeological connection on any principal $G$-bundle.  By Remark~\ref{r:homotopy}, Proposition~\ref{p:homotopy} holds without any assumption on the topology of the base.  Since Proposition~\ref{p:uniqueness} also has no conditions on the topology of the base, it follows that \emph{if} you have two principal $G$-bundles (with no assumptions on them) $E\to X$ and $E'\to X$ \emph{with classifying maps} $F\colon X\to BG$ and $F'\colon X\to BG$, then $E$ and $E'$ are isomorphic bundles if and only if $F$ and $F'$ are smoothly homotopic.  
\end{remark}

%%%%%%%%%%%%%%%%%%%%%%%%%%%%%%%%%%%%%%%%%%%%%%%%%%%%%%%%%%%%%%%%%%%%%%%%
\section{Applications}\labell{s:applications}
%%%%%%%%%%%%%%%%%%%%%%%%%%%%%%%%%%%%%%%%%%%%%%%%%%%%%%%%%%%%%%%%%%%%%%%%

In this section, we apply the theory developed in the previous sections to various situations.  Many of these are motivated by applications found in \cite{iglesias}, \cite{KM}, \cite{Ma2013}, \cite{Ma2016}, \cite{R}, and \cite{tD}.

%% %% %%
\subsection{Contractibility of $EG$ \& Homotopy of $BG$}\labell{ss:contractible}
%% %% %%

In this subsection, we show that $EG$ is smoothly contractible, which allows us to compute the (diffeological) homotopy groups of $BG$ in terms of those of $G$.  Again, we look to the topological proof found in \cite{tD} for a template.

\begin{proposition}[$EG$ is Contractible]\labell{p:contractible}
Let $G$ be a diffeological group.  Then $EG$ is smoothly contractible.
\end{proposition}

\begin{proof}
Let $H_1\colon EG\times[0,1]\to EG$ be a smooth homotopy given by 
\begin{gather*}
H_1((t_ig_i),0)=(t_ig_i) \text{ and}\\
H_1((t_ig_i),1)=(t_1g_1,0,t_2g_2,t_3g_3,\dots). 
\end{gather*}
using the same notation as that in the proof of Proposition~\ref{p:uniqueness}.  This exists by Proposition~\ref{p:uniqueness}.  Let $b$ be the smooth function from the same proof and let $H_2\colon EG\times[0,1]\to EG$ be the smooth map 
$$H_2((t_ig_i),\tau)=\big((1-b(\tau))t_1g_1,b(\tau)e,(1-b(\tau))t_2g_2,(1-b(\tau))t_3g_3,\dots\big)$$

Concatenating $H_1$ and $H_2$ yields a smooth homotopy.  Since $$H_2((t_ig_i),1)=\big(0,e,0,0,\dots\big),$$ which is constant on $EG$, it follows that $EG$ is smoothly contractible.
\end{proof}

\begin{proposition}[Homotopy of $BG$]\labell{p:homotopy of BG}
Let $G$ be a diffeological group.  Then for each $k>0$, we have $\pi_k(BG)\cong\pi_{k-1}(G)$.
\end{proposition}

\begin{proof}
This is immediate from the fact that $EG\to BG$ is a principal $G$-bundle, Proposition~\ref{p:contractible}, and the long exact sequence of (diffeological) homotopy groups; see \cite[Article 8.21]{iglesias}.
\end{proof}

\begin{example}[Homotopy of $BT_\alpha$]\labell{x:irrational torus}
Fix an irrational number $\alpha$.  Let $G$ be the irrational torus $T_\alpha:=\RR/\ZZ^2$, where $\ZZ^2$ acts on $\RR$ by $(m,n)\cdot x=x+m+n\alpha$.  By the long exact sequence of (diffeological) homotopy groups \cite[Article 8.21]{iglesias}, it is immediate that $\pi_0(T_\alpha)=\{T_\alpha\}$, $\pi_1(T_\alpha)=\ZZ^2$, and $\pi_k(T_\alpha)=0$ for $k>1$.  It follows from Proposition~\ref{p:homotopy of BG} that $\pi_0(BT_\alpha)=\{BT_\alpha\}$, $\pi_1(BT_\alpha)=0$, $\pi_2(BT_\alpha)=\ZZ^2$, and $\pi_k(BT_\alpha)=0$ for all $k>2$.
\eoe
\end{example}

%% %% %%
\subsection{Smooth Homotopies of Groups and Bundles}\labell{ss:homotopies}
%% %% %%

Here we look at how smooth group homomorphisms between diffeological groups induce smooth maps between the classifying spaces, with applications to smooth strong deformation retractions.

\begin{proposition}[Smooth Maps Between Diffeological Groups]\labell{p:maps between gps}
Let $G$ and $H$ be diffeological groups, and let $\varphi\colon G\to H$ be a smooth map between them.  Then $\varphi$ induces a smooth map $\widetilde{\Phi}\colon EG\to EH$ defined as $$\widetilde{\Phi}(t_ig_i)=(t_i\varphi(g_i)).$$  Moreover, if $\varphi$ is a smooth homomorphism, $\widetilde{\Phi}$ descends to a smooth map $\Phi\colon BG\to BH$.
\end{proposition}

\begin{proof}
The proof follows immediately from the definitions.
\end{proof}

\begin{remark}\labell{r:maps between gps}
If $\varphi_\tau\colon G\to H$ is a smooth family of maps between diffeological groups $G$ and $H$, then a similar proof to the above yields a smooth family of maps $\widetilde{\Phi}_\tau\colon EG\to EH$ defined in the obvious way, and in the case of a smooth family of group homomorphisms, we obtain a smooth family of maps $\Phi_{\tau}\colon BG\to BH$.
\end{remark}

\begin{corollary}[Smooth Deformation Retracts of Bundles]\labell{c:maps between gps}
Let $G$ be a diffeological group, and let $\Phi\colon G\times[0,1]\to G$ be a smooth strong deformation retraction of $G$ onto a subgroup $H$ such that for each $\tau\in[0,1]$, the map $\Phi(\cdot,\tau)\colon G\to G$ is a group homomorphism.  Then, up to isomorphism, there is a smooth strong deformation retract of any weakly D-numerable principal $G$-bundle over a Hausdorff smoothly paracompact base to a principal $H$-bundle over the same base.
\end{corollary}

\begin{proof}
By Proposition~\ref{p:maps between gps} and Remark~\ref{r:maps between gps}, we obtain a smooth strong deformation retract $\widetilde{\Phi}_\tau$ of $EG$ onto $EH$, which descends to a smooth strong deformation retract $\Phi_\tau$ of $BG$ onto $BH$.  Let $E\to X$ be a weakly D-numerable principal $G$-bundle over a Hausdorff smoothly paracompact base $X$.  By Proposition~\ref{p:classifying maps}, there is a smooth map $\widetilde{F}\colon E\to EG$ which descends to a smooth classifying map $F\colon X\to BG$ for which $E$ is isomorphic as a principal $G$-bundle to $F^*EG$.  Composing $F$ with $\Phi_\tau$, we obtain a smooth strong deformation retract of $F^*EG$ onto $(\Phi_0\circ F)^*EH$, where at any $\tau$ we have $$(\Phi_\tau\circ F)^*EG=\{(x,(t_ig_i))\mid \Phi_\tau\circ F(x)=[t_ig_i]\}.$$
\end{proof}

\begin{example}[$E(\Diff(\RR^n;0))$]\labell{x:diff Rn}
Let $G=\Diff(\RR^n;0)$ be the diffeological group of diffeomorphisms of $\RR^n$ that fix the origin.  Define $\Phi\colon G\times[0,1]\to G$ by 
$$\Phi(\varphi,\tau)=\begin{cases}
m_{1/\tau}\circ\varphi\circ m_\tau & \text{if $\tau\neq 0$,}\\
d\varphi|_0 & \text{if $\tau=0$,}\\
\end{cases}$$
where $m_\tau\colon\RR^n\to\RR^n$ is scalar multiplication by $\tau$ (which is smooth).
By definition of the functional diffeology \cite[Article 1.57]{iglesias}, it is an easy exercise to check that this is a smooth strong deformation retract of $G$ onto $\GL(n;\RR)$.  Moreover, the chain rule shows that $\Phi(\cdot,\tau)$ is a group homomorphism for each $\tau$.  It follows from Proposition~\ref{p:maps between gps}  that $EG$ has a smooth strong deformation retract onto $E\GL(n;\RR)$, and this descends to a smooth strong deformation retract of $BG$ onto $B\GL(n;\RR)$.  By Corollary~\ref{c:maps between gps} we have that any weakly D-numerable principal $G$-bundle over a Hausdorff smoothly paracompact base has a smooth strong deformation retract onto a principal $\GL(n;\RR)$-bundle over the same base.
\eoe
\end{example}

\begin{example}[$E(\Diff^+(\SS^2))$]\labell{x:diff S2}
Let $G=\Diff^+(\SS^2)$ be the diffeological group of orientation-preserving diffeomorphisms of $\SS^2$ (equipped with the functional diffeology).  By the main result of \cite{LW2011} there is a smooth strong deformation retraction $\kappa\colon G\times[0,1]\to G$ onto $\SO(3)$.  By Proposition~\ref{p:maps between gps} and Remark~\ref{r:maps between gps}, $\kappa$ induces a smooth strong deformation retraction from $EG$ to $E(\SO(3))$.  Unfortunately, the smooth strong deformation retraction in \cite{LW2011} does not give a group homomorphism from $G$ to $\SO(3)$, and so the deformation retraction does not descend to $BG$ and $B\SO(3)$.
\eoe
\end{example}

%% %% %%
\subsection{Associated Fibre Bundles}\labell{ss:associated bundles}
%% %% %%

Let $G$ be a diffeological group, and let $\pi\colon E\to X$ be a principal $G$-bundle.  Let $F$ be a diffeological space admitting a smooth action of $G$.  Then we may construct the \textbf{associated bundle to $\pi$ with fibre $F$}, denoted $\widetilde{\pi}\colon \widetilde{E}\to X$, as follows: let $G$ act diagonally on $E\times F$; then define $\widetilde{E}$ as the quotient $E\times_G F:=(E\times F)/G$.  Any diffeological fibre bundle with structure group $G$ can be constructed as an associated fibre bundle of some principal $G$-bundle.  See \cite[Article 8.16]{iglesias} for more details.  Note that if a fibre bundle is (weakly) D-numerable any associated principal $G$-bundle is (weakly) D-numerable as well.  We have the following corollary to Theorem~\ref{t:BG}:

\begin{corollary}[Diffeological Fibre Bundles]\labell{c:BG}
Let $\mathcal{B}^F_G(\cdot)$ be the functor from $\mathbf{Diffeol}_{\operatorname{HSP}}$ to $\mathbf{Set}$ sending an object $X$ to the set of isomorphism classes of locally trivial diffeological fibre bundles with fibre a fixed $G$-space $F$.  Then there is a natural surjection from $[\cdot,BG]$ to $\mathcal{B}^F_G(\cdot)$.
\end{corollary}

Using the same notation as above, let $\theta$ be a diffeological connection on $E$ induced by a connection 1-form $\omega$ on $\pi\colon E\to X$ as in Proposition~\ref{p:iz-connection}.  Now $\omega\oplus 0$ is a $T_{e}G$-valued 1-form on $E\times F$, and one can check using the fact that $\pi$ is D-numerable that $\omega$ descends to a $T_{e}G$-valued 1-form $\omega_{\widetilde{E}}$ on $\widetilde{E}$, which we call a \textbf{connection 1-form} on $\widetilde{E}$.

\begin{definition}[Horizontal Lifts on Associated Bundles]\labell{d:assoc lifts}
Let $\widetilde{\pi}\colon \widetilde{E}\to X$ be a diffeological fibre bundle, let $c\colon(a,b)\to X$ be a smooth curve, and fix $z\in\widetilde{E}$ and $t_0\in(a,b)$ such that $\widetilde{\pi}(z)=c(t_0)$.  A curve $c_{\widetilde{E}}\colon(a,b)\to \widetilde{E}$ is a \textbf{horizontal lift} of $c$ through $z$ if
\begin{enumerate}
\item $\widetilde{\pi}\circ c_{\widetilde{E}}=c$,
\item\labell{i:initial condition} $c_{\widetilde{E}}(t_0)=z$,
\item\labell{i:assoc lifts de} $\omega_{\widetilde{E}}\left(\frac{d}{dt}c_{\widetilde{E}}(t)\right)=0.$
\end{enumerate}
\end{definition}

\begin{proposition}[Horizontal Lifts on Associated Bundles Exist]\labell{p:assoc lifts}
Let $G$ be a regular diffeological Lie group, $F$ a diffeological space with a fixed smooth action of $G$, $X$ a Hausdorff smoothly paracompact diffeological space, and $\pi\colon E\to X$ a weakly D-numerable principal $G$-bundle.  Equip $E$ with the connection 1-form $\omega$ as in Corollary~\ref{c:universal connection}.  Then for any smooth curve $c\colon(a,b)\to X$, and any $z\in\widetilde{E}$, there is a horizontal lift $c_{\widetilde{E}}\colon(a,b)\to\widetilde{E}$ through $z$.
\end{proposition}

\begin{proof}
Take $c_{\widetilde{E}}(t):=\rho_{\widetilde{E}}(\theta(c,t_0)(t),z')$, where $\rho_{\widetilde{E}}\colon E\times F\to\widetilde{E}$ is the quotient map and $z'\in F$ such that $\rho_{\widetilde{E}}(\theta(c,t_0)(t_0),z')=z$.  This definition is independent of the choice of $z'$.  This satisfies all of properties required in Definition~\ref{d:assoc lifts}.
\end{proof}

\begin{remark}\labell{r:assoc lifts}
We do not mention uniqueness of horizontal lifts in Proposition~\ref{p:assoc lifts}.  This requires the uniqueness of solutions to the initial-value problem given by the differential equation in Item~(\ref{i:assoc lifts de}) with initial condition as given in Item~(\ref{i:initial condition}), both in Definition~\ref{d:assoc lifts}.  It is not immediately clear under what conditions this would hold on a general diffeological space.
\end{remark}

\begin{example}[$\Diff(F)$-Bundles]\labell{x:assoc bdles}
Let $\pi\colon \widetilde{E}\to X$ be a diffeological fibre bundle with fibre $F$ and structure group $G=\Diff(F)$.  Define the set $\mathcal{FR}$ to be all commutative diagrams 
$$\xymatrix{
F \ar[r]^f \ar[d] & \widetilde{E} \ar[d]^{\pi} \\
\{*\} \ar[r] & X \\
}$$
where $f$ is a diffeomorphism onto a fibre of $\widetilde{E}$.  Equip $\mathcal{FR}$ with the subset diffeology induced by $\CIN(F,\widetilde{E})$.  One may think of this as a ``non-linear frame bundle'' of the bundle $\widetilde{E}$.  From above, we have that $\mathcal{FR}\times_{\Diff(F)}F\cong \widetilde{E}$, and by Corollary~\ref{c:BG}, we have that there is a natural surjection from smooth homotopy classes of maps $X\to B\Diff(F)$ to isomorphism classes of such bundles $\widetilde{E}$, provided that the bundles are locally trivial, and $X$ has Hausdorff, second-countable, and smoothly paracompact topology.

If $F$ is a smooth manifold, then $\Diff(F)$ is a regular diffeological Lie group \cite[Theorem 43.1]{KM}, and so provided that $\widetilde{E}\to X$ is a locally trivial principal $\Diff(F)$-bundle, and the topology on $X$ is Hausdorff, second-countable, and smoothly paracompact, then by Proposition~\ref{p:assoc lifts} $\widetilde{E}$ has a diffeological connection, which in turn gives us horizontal lifts of curves.
\eoe
\end{example}

%% %% %%
\subsection{Limit of Groups \& ILB Principal Bundles}\labell{ss:proj limits}
%% %% %%

We consider in this subsection limits of diffeological group, and in particular infinite-dimensional groups.  We rely heavily on \cite{Om} for terminology.

Fix a small category $J$, which we will think of as our ``index category''.  Let $\mathbf{DGroup}$ be the category of diffeological groups with smooth group homomorphisms between them, and let $F\colon J\to\mathbf{DGroup}$ be a functor.  Denote by $G_j$ the image $F(j)$ for each object $j$ of $J$, and by $\varphi_f$ the image $F(f)$ for each arrow $f$ in $J$.  Let $G=\lim F$ be the limit taken in the category of diffeological spaces; in particular, there is a smooth map $\varphi_j\colon G\to G_j$ for each object $j$ in $J$ such that if $f\colon j_1\to j_2$ is an arrow in $J$, then $\varphi_f\circ\varphi_{j_1}=\varphi_{j_2}.$

\begin{proposition}[Limit of Diffeological Groups]\labell{p:limit d-groups}
Let $F$ be the functor above, and define $G=\lim F$.  Then $G$ is a diffeological group.
\end{proposition}

\begin{proof}
Recall that both the category of diffeological spaces and the category of groups are complete, and limits are constructed on the set-theoretical level: $$G=\left\{(g_j)\in\prod_{j\in J_0}G_j~\Biggr|~\forall(f\colon j_1\to j_2)\in J_1,~\varphi_f\circ\pr_{j_1}((g_j))=\pr_{j_2}((g_j))\right\}$$ where $J_0$ is the set of objects of $J$, $J_1$ the set of arrows, and $\pr_j$ is the $j^{\text{th}}$ projection onto $G_j$.  The diffeology is given by the intersection of all pullback diffeologies on $G$ via the maps $\varphi_j:=\pr_j$.  The group multiplication on $G$ is given by the coordinate-wise product: $(g_j)(h_j):=(g_jh_j)$.  It follows that the maps $\varphi_j$ are group homomorphisms, from which it follows that multiplication and inversion are smooth maps $G\times G\to G$ and $G\to G$, respectively.  Hence, $G$ is a diffeological group.
\end{proof}

\begin{remark}\labell{r:limit d-groups}
Let $G$ be the diffeological group constructed as a limit above.  Then for every object $j$ of $J$, from Proposition~\ref{p:maps between gps} we obtain a commutative diagram:

$$\xymatrix{
EG \ar[r]^{\widetilde{\Phi}_j} \ar[d] & EG_j \ar[d] \\
BG \ar[r]_{\Phi_j} & BG_j. \\
}$$

Moreover, for any arrow $f\colon j_1\to j_2$ of $J$, we obtain a commutative diagram:

$$\xymatrix{
EG \ar[rr] \ar[dr] \ar[ddd] & & EG_{j_1} \ar[dl]^{\widetilde{\Phi}_{f}} \ar[ddd] \\
 & EG_{j_2} \ar[d] & \\
 & BG_{j_2} & \\
BG \ar[rr] \ar[ur] & & BG_{j_1} \ar[ul]_{\Phi_{f}} \\
}$$
where all of the maps are the induced ones via Proposition~\ref{p:maps between gps}.
\end{remark}

We now apply the theory developed to ILB-manifolds and ILB-bundles.  Technically, our definition of an ILB-manifold below is only a special case of the definition given in \cite[Definition I.1.9]{Om}; however, it is sufficient for our needs.  For this application to make sense, we must first note that infinite-differentiability in the Fr\'echet sense is equivalent to smoothness in the diffeological sense when we equip a Fr\'echet space/manifold with the diffeology comprising Fr\'echet infinitely-differentiable parametrisations.  Moreover, Fr\'echet spaces form a full subcategory of diffeological spaces under this identification.  See, for example, \cite[Theorem 3.1.1]{losik}, \cite{KaWa}.  Note that the limits used in the following definitions can be taken in either the Fr\'echet category or the diffeological category: in our cases below they coincide.

\begin{definition}[ILB-Manifolds \& Bundles]\labell{d:ilb}
\noindent
\begin{enumerate}

\item An \textbf{ILB-manifold} $M$ is the limit in the category of Fr\'echet manifolds of a family of Banach manifolds $\{M_n\}_{n \in \NN}$ in which there is a smooth and dense inclusion $M_{n+1}\hookrightarrow M_n$ for each $n$, and such that there exists a Banach atlas on $M_0$ that restricts to an atlas on $M_n$ for each $n$.

\item An \textbf{ILB-map} $f$ between two ILB-manifolds $M$ and $N$ is a smooth map $f\colon M\to N$ along with a family of smooth maps $\{f_n\colon M_n\to N_n\}$ such that $f$ and all $f_n$ commute with all inclusions maps $M\hookrightarrow M_{n+1}\hookrightarrow M_n$ and $N\hookrightarrow N_{n+1}\hookrightarrow N_n$.

\item An \textbf{ILB-principal bundle} is an ILB-map between two ILB-manifolds $(\pi\colon P\to M,~\pi_n\colon P_n\to M_n)$ such that for each $n$, the map $\pi_n\colon P_n\to M_n$ is a principal $G_n$-bundle where $G_n$ is a Banach Lie group.

\item An ILB-map $F$ between two ILB-principal bundles $\pi\colon P\to M$ and $\pi'\colon P'\to M'$ is an \textbf{ILB-bundle map} if $F_n\colon P_n\to P'_n$ is a ($G_n$-equivariant) bundle map for each $n$.  Note that $F$ induces an ILB-map $M\to M'$.  An ILB-bundle map is an \textbf{ILB-bundle isomorphism} if it has an inverse ILB-bundle map.

\end{enumerate}
\end{definition}

It follows from the above definitions that given an ILB-principal bundle $\pi\colon P\to M$, the structure groups $\{G_n\}$ of the principal bundles $\{\pi_n\colon P_n\to M_n\}$ satisfy: $G_{n+1}$ is smoothly and densely included in $G_n$ for each $n$.  Hence $\pi\colon P\to M$ is a principal $G$-bundle where $G=\lim G_n$.  An ILB-bundle map from $\pi\colon P\to M$ to $\pi'\colon P'\to M'$ is thus necessarily $G$-equivariant.

\begin{definition}[ILB-Classifying Maps \& Homotopies]\labell{d:ilb2}
\noindent
\begin{enumerate}
\item Given an ILB-principal bundle $\pi\colon P\to M$ with structure group $G$, an \textbf{ILB-classifying map} $F\colon M\to BG$ is a classifying map for $\pi$ such that there is a classifying map $F_n\colon M_n\to BG_n$ for each $n$, the following diagram commutes for all $n$,

$$\xymatrix{
P \ar[rrr] \ar[dr]_{\widetilde{F}} \ar[ddd] & & & P_n \ar[dl]^{\widetilde{F}_n} \ar[ddd] \\
 & EG \ar[r]^{\widetilde{\Phi}_n} \ar[d] & EG_n \ar[d] & \\
 & BG \ar[r]_{\Phi_n} & BG_n & \\
M \ar[rrr] \ar[ur]^{F} & & & M_n \ar[ul]_{F_n} \\
}$$
and these maps commute with the inclusion maps $EG\hookrightarrow EG_{n+1}\hookrightarrow EG_n$ and $BG\hookrightarrow BG_{n+1}\hookrightarrow BG_n$ as in Remark~\ref{r:limit d-groups}.

\item An \textbf{ILB-homotopy} between two ILB-classifying maps $F,F'\colon M\to BG$ for ILB-principal bundles $P\to M$ and $P'\to M$ each with structure group $G$ is a family of smooth homotopies $H_n\colon M_n\times[0,1]\to BG_n$ such that $H_n(\cdot,0)=F_n$ and $H_n(\cdot,1)=F'_n$ for each $n$, and a smooth homotopy $H\colon M\times[0,1]\to BG$ such that $H(\cdot,0)=F$ and $H(\cdot,1)=F'$, and finally each $H_n$ and $H$ commute with the maps in Remark~\ref{r:limit d-groups}.

\end{enumerate}
\end{definition}

Let $\pi\colon P\to M$ be an ILB-principal bundle, and assume that each $M_n$ is a Hilbert manifold (in this case, we say that $M$ is an \textbf{ILH-manifold}).  Hilbert manifolds are smoothly paracompact \cite[Corollary 16.16]{KM}, and (smooth) partitions of unity pull back.  Hence, a smooth partition of unity subordinate to a local trivialisation of $\pi_0\colon P_0\to M_0$ induces similar partitions of unity for each $\pi_n\colon P_n\to M_n$, and on $\pi\colon P\to M$.  Thus, in this case, each $\pi_n\colon P_n\to M_n$, as well as $\pi\colon P\to M$, is D-numerable.

Given an ILB-principal bundle $\pi\colon P\to M$ with ILB group $G$ and ILH base as above, by Proposition~\ref{p:classifying maps} there is a classifying map $F\colon M\to BG$ and so $\pi\colon P\to M$ is isomorphic to $F^*EG$.  Moreover, for each $n$, there is a classifying map $F_n\colon M_n\to BG_n$ and so $P_n$ is isomorphic to $F_n^*EG_n$.  It follows from Remark~\ref{r:limit d-groups} that we obtain an ILB-classifying map for $\pi\colon P\to M$.  From Theorem~\ref{t:BG} and Remark~\ref{r:limit d-groups} we obtain the following proposition:

\begin{proposition}\labell{p:ilb}
Let $\pi\colon P\to M$ be an ILB-principal $G$-bundle in which $M=\lim M_n$ is an ILH-manifold.  Then $\pi$ has an ILB-classifying map.  Moreover, smoothly homotopic ILB-classifying maps yield isomorphic ILB-principal bundles yield ILB-homotopic ILB-classifying maps.
\end{proposition}

%% %% %%
\subsection{Group Extensions}\labell{ss:extensions}
%% %% %%  
  
Fix a diffeological extension of diffeological groups (see \cite[Article 7.3]{iglesias}):
\begin{equation}\labell{e:ses}
1 \rightarrow G \rightarrow G' \rightarrow G'' \rightarrow 1;
\end{equation}
 By Proposition~\ref{p:maps between gps}, we obtain the following commutative diagram of diffeological spaces.

$$\xymatrix{
EG \ar[r] \ar[d] & EG' \ar[r] \ar[d] & EG'' \ar[d] \\
BG \ar[r] & BG' \ar[r] & BG'' \\
}$$

Since diffeological extensions are examples of principal bundles, \eqref{e:ses} induces a long exact sequence of diffeological homotopy groups (see \cite[Article 8.21]{iglesias}).  The following result thus is a consequence of Proposition~\ref{p:homotopy of BG}.

\begin{proposition}[Long Exact Sequence of Classifying Spaces]\labell{p:les}
Given a diffeological extension of diffeological groups as in \eqref{e:ses}, we obtain a long exact sequence
$$\dots \to \pi_{n+1}(BG'') \to \pi_n(BG) \to \pi_n(BG') \to \pi_n(BG'') \to \pi_{n-1}(BG) \to \dots~.$$
\end{proposition}

\begin{example}[$\RR/\QQ$]\labell{x:R/Q}
Consider the rational numbers $\QQ$ as a diffeologically discrete subgroup of $\RR$.  We have the short exact sequence of diffeological groups $$1 \to \QQ \to \RR \to \RR/\QQ \to 1.$$  From the above, we get a long exact sequence of homotopy groups:
$$\dots \to \pi_{n+1}(B(\RR/\QQ)) \to \pi_n(B\QQ) \to \pi_n(B\RR) \to \pi_n(B(\RR/\QQ)) \to \pi_{n-1}(B\QQ) \to \dots~.$$

It follows from this long exact sequence and Proposition~\ref{p:homotopy of BG} that

\begin{gather*}
\pi_0(B(\RR/\QQ))=\{B(\RR/\QQ)\},\\
\pi_1(B(\RR/\QQ))\cong 1,\\
\pi_2(B(\RR/\QQ))\cong\QQ, \text{ and}\\
\pi_k(B(\RR/\QQ))\cong 1 \text{ for all $k\geq 3$}.
\end{gather*}

Of course, we could have used other easier means of computing these, but the point of this example is to illustrate what one could do with Proposition~\ref{p:les}.
\eoe
\end{example}

Using short exact sequences and group extensions is very frequent in the study of infinite-dimensional Lie groups for constructing new groups. For example:

\begin{example}[$\Diff(M)$-Pseudo-Differential Operators]\labell{x:pdo}
Consider the $\Diff(M)$-pseudo-differential operators $FIO_{\Diff}^{0,*}$ described in \cite{Ma2016}. Given a Hermitian bundle $E \rightarrow M$ over a closed manifold $M,$ the group $FIO_{\Diff}^{0,*}(M,E)$ is a group of Fourier integral operators acting on smooth sections of $E$, which can be seen as a central extension of the group of diffeomorphisms $\Diff(M)$ by the group of $0$-order invertible pseudo-differential operators acting on smooth sections of $E,$ with exact sequence:
$$ 0 \rightarrow PDO^{0,*}(M,E) \rightarrow FIO_{\Diff}^{0,*}(M,E) \rightarrow \Diff(M) \rightarrow 0.$$
We obtain a long exact sequence of homotopy groups:
\begin{gather*}
\dots \to \pi_{n+1}(B\Diff(M)) \to \pi_n\left(B(PDO^{0,*}(M,E))\right) \to \pi_n\left(B(FIO_{\Diff}^{0,*}(M,E))\right) \to \\
\pi_n(B\Diff(M)) \to \pi_{n-1}\left(B(PDO^{0,*}(M,E))\right) \to \dots~.
\end{gather*}
\eoe
\end{example}

\subsection{Irrational Torus Bundles}\labell{ss:irrational torus}

We now connect what we have done to the classical theory of Weil, in which circle (or complex line) bundles with connection over a fixed simply-connected manifold are classified by their curvatures, and conversely every integral 2-form on the manifold induces a circle bundle with connection whose curvature is that 2-form.  This is extended to non-integral 2-forms in \cite{iglesias-bdles}, in which the circle bundles are replaced with irrational torus bundles.  See also the more general theory on diffeological spaces in \cite{iglesias}.

Let $T_\alpha$ be the irrational torus, defined in Example~\ref{x:irrational torus}.  It is a regular diffeological Lie group (indeed, the quotient map $\RR\to T_\alpha$ is the exponential map).  Let $\omega$ be the connection 1-form on $ET_\alpha$ constructed in Theorem~\ref{t:universal connection}, and let $\pi\colon ET_\alpha\to BT_\alpha$ be the projection map.

\begin{lemma}\labell{l:curvature}
The 2-form $d\omega$ is basic; that is, there exists a unique $T_{e}G$-valued 2-form $\Omega$ on $BT_\alpha$ such that $\pi^*\Omega=d\omega$.
\end{lemma}

\begin{remark}\labell{r:curvature}
This statement in fact holds for any abelian diffeological group $G$.  We prove this more general statement below.
\end{remark}

\begin{proof}
We again use the following fact (see \cite[Article 6.38]{iglesias}): $d\omega$ is the pullback of a form $\Omega$ on $BG$ if and only if for any two plots $p_1,p_2:U\to EG$ satisfying $\pi\circ p_1=\pi\circ p_2$, we have $p_1^*d\omega=p_2^*d\omega$.  Fix two such plots $p_1$ and $p_2$.  Let $V_j\subseteq EG$ be the open set $$V_j=\{(t_ig_i)\mid t_j\neq 0\}.$$  We have $p_1^{-1}(V_j)=p_2^{-1}(V_j)=:V$.  The projection $\pr_{g_j}\colon V_j\to G$ sending $(t_ig_i)$ to $g_j$ is well-defined on $V_j$, and so we have a smooth map $\gamma\colon V\to G$ sending $u$ to $\pr_{g_j}(p_2(u))^{-1}\pr_{g_j}(p_1(u))$.  Thus, $\gamma(u)\cdot p_2(u)=p_1(u)$ for all $u\in V$.

Fix $u\in V$ and $v\in T_uV$.  It follows from the chain rule (Remark~\ref{r:tangent bdle}) that $$(p_1)_*v\hook d\omega=(p_2)_*v\hook(\gamma(u)^*d\omega)+\eta\hook d\omega,$$ where $\eta=\frac{d}{dt}\Big|_{t=0}(\gamma(c(t))\cdot p_1(u))$ for some smooth curve $c\colon(-\eps,\eps)\to V$ such that $c(0)=u$ and $\dot{c}(0)=v$ ($\eps>0$).  Since $\omega$ is invariant, we have $\gamma(u)^*d\omega=d\omega$.

Since $j$, $u$, and $v$ are arbitrary, to complete the proof, we only need to show that $\eta\hook d\omega=0$.  Note that this is equal to the \emph{contraction} of $d\omega$ by the map $g_{EG}\colon(-\eps,\eps)\to\Diff(EG)$ induced by the curve $g:=\gamma\circ c$ in $G$ (see \cite[Article 6.57]{iglesias}).  By the Cartan-Lie formula (\cite[Article 6.72]{iglesias}), we have 
\begin{equation}\labell{e:curvature1}
\eta\hook d\omega=\pounds_g(\omega)-d(\eta\hook\omega).
\end{equation}
By definition of $\omega$ we know $\eta\hook\omega=\dot{g}(0)$, and so the right term of the right-hand side of \eqref{e:curvature1} vanishes.  The left term of the right-hand side of \eqref{e:curvature1} vanishes since $\omega$ is invariant (see \cite[Article 6.55]{iglesias} for a definition of the Lie derivative).

Uniqueness of $\Omega$ follows from the fact that $\pi^*$ is injective on forms.  This finishes the proof.
\end{proof}

We refer to $\Omega$ in Lemma~\ref{l:curvature} as the \textbf{curvature form} of $\omega$.  If $X$ is any Hausdorff smoothly paracompact diffeological space, and $F\colon X\to BT_\alpha$ a smooth function, we obtain the \textbf{curvature form} $F^*\Omega$ of the connection 1-form $\widetilde{F}^*\omega$ on $F^*EG$ by Proposition~\ref{p:classifying maps}.  The 2-form $F^*\Omega$ satisfies $$\pr_1^*F^*\Omega=\widetilde{F}^*d\omega$$ where $\pr_1\colon F^*EG\to X$ is the pullback bundle.

Let $X$ be a connected diffeological space.  In \cite[Article 8.40]{iglesias}, Iglesias-Zemmour proves that every principal $T_\alpha$-bundle which can be equipped with a connection 1-form induces a unique class in $H^1(\paths(X),T_\alpha)$, its \textbf{characteristic class}.  If additionally $X$ Hausdorff and smoothly paracompact, then it follows from Proposition~\ref{p:classifying maps} that \emph{all} weakly D-numerable $T_\alpha$-bundles over $X$ induce a unique class in $H^1(\paths(X),T_\alpha)$.

In \cite[Article 8.42]{iglesias}, Iglesias-Zemmour shows a converse for the simply-connected case: if $X$ is a simply-connected (and hence connected) diffeological space, and $\mu$ is a non-zero closed 2-form on $X$, then there is a principal $T_\alpha$-bundle on $X$ whose curvature is $\mu$, where $T_\alpha$ is the \textbf{torus of periods} of $\mu$.  If $X$ is additionally Hausdorff and smoothly paracompact, and the $T_\alpha$-bundle constructed is weakly D-numerable, then $\mu=F^*\Omega$, where $F$ is the classifying map of the $T_\alpha$-bundle, and $\Omega$ is the curvature form on $BT_\alpha$.

%%%%%%%%%%%%%%%%%%%%%%%%%%%%%%%%%%%%%%%%%%%%%%%%%%%%%%%%%%%%


\begin{thebibliography}{99}

\bibitem[A]{awodey}
Steve Awodey \emph{Category theory, Oxford Logic Guides, 49}, The Clarendon Press, Oxford University Press, New York, 2006.

%\bibitem[BIZKW]{BIZKW}
%Augustin Batubenge, Patrick Iglesias-Zemmour, Yael Karshon, and Jordan Watts, ``Diffeological, Fr\"olicher, and differential spaces'' (in progress).\\ \url{http://euclid.colorado.edu/~jowa8403/papers/reflexive.pdf}

%\bibitem[BH]{BH}
%John C.\ Baez and Alexander E.\ Hoffnung, ``Convenient categories of smooth spaces'', \emph{Trans. Amer. Math. Soc.}, \textbf{363} (2011), 5789--5825.

%\bibitem{Bry} Brylinski, J-L.; \textit{Loop spaces, characteristic classes and geometric quantization} Modern Birkhauser classics(2008)

\bibitem[CSW]{CSW}
J.\ Daniel Christensen, Gordon Sinnamon, Enxin Wu, ``The $D$-topology for diffeological spaces'', \emph{Pacific J.\ Math.\ }\textbf{272} (2014), 87--110.

\bibitem[CW16]{CW}
J.\ Daniel Christensen and Enxin Wu, ``Tangent spaces and tangent bundles for diffeological spaces'', \emph{Cahiers de Topologie et G\'eom\'etrie Diff\'erentielle} 57 (2016), 3--50.

\bibitem[CW17]{CW2}
J.\ Daniel Christensen and Enxin Wu, ``Smooth classifying spaces'', (preprint).\\
\url{https://arxiv.org/abs/1709.10517}

\bibitem[Cr]{crainic}
Marius Crainic, ``Prequantization and Lie brackets'', \emph{J.\ Symplectic Geom.} \textbf{2} (2004), 579--602.

\bibitem[I85]{IgPhD} Patrick Iglesias, \emph{Fibrations diff\'eologiques et Homotopie},  PhD thesis, Universit\'e de Provence (1985).\\
\url{http://math.huji.ac.il/~piz/documents/TheseEtatPI.pdf}

\bibitem[I95]{iglesias-bdles}
Patrick Iglesias, ``La trilogie du moment'', \emph{Ann.\ Inst.\ Fourier (Grenoble)} \textbf{45} (1995), 825--857.

\bibitem[IZ13]{iglesias}
Patrick Iglesias-Zemmour, \emph{Diffeology, Math.\ Surveys and Monographs 185}, Amer.\ Math.\ Soc., 2013.

\bibitem[KM]{KM}
Andreas Kriegl and Peter W.\ Michor, \emph{The Convenient Setting for Global Analysis, Math.\ Surveys and Monographs 53}, Amer.\ Math.\ Soc., 1997.

\bibitem[KhWe]{KW}
Boris Khesin and Robert Wendt, \emph{The Geometry of Infinite-Dimensional Groups, Ergebnisse der Mathematik und ihrer Grenzgebiete. 3. Folge. A Series of Modern Surveys in Mathematics, 51}, Springer-Verlag, Berlin, 2009.

\bibitem[KaWa]{KaWa}
Yael Karshon and Jordan Watts, ``Diffeological Spaces, Differential Spaces, and Reflexivity'' (tentative title), in progress.

\bibitem[Le]{leslie}
Joshua Leslie, ``On a diffeological group realization of certain generalized symmetrizable Kac-Moody Lie algebras'', \emph{J.
Lie Theory} \textbf{13} (2003), 427--442.

\bibitem[LW]{LW2011}
Jiayong Li and Jordan Watts, ``The orientation-preserving diffeomorphism group of $\SS^1$ deforms to $SO(3)$ smoothly'', \emph{Transformation Groups} \textbf{16} (2011), 537--553.

\bibitem[Lo]{losik}
Mark V.\ Losik, ``Categorical differential geometry'', \emph{Cah.\ Topol.\ G\'eom.\ Diff\'er.\ Cat\'eg.}, \textbf{35} (1994), 274--290.

%\bibitem[Ma06]{Ma2006}
%Jean-Pierre Magnot, ``Diff\'eologie sur le fibr\'e d'holonomie d'une connexion en dimension infinie'', \emph{C.\ R.\ Math.\ Acad.\ Sci., Soc.\ R.\ Can.} \textbf{28} (2006), 121--127.

\bibitem[Ma11]{Ma2011}
Jean-Pierre Magnot, ``A non regular Fr\"olicher Lie group of diffeomorphisms'', 5 pages (preprint)\\
\url{http://arxiv.org/pdf/1101.2370v4.pdf}

\bibitem[Ma13]{Ma2013}
Jean-Pierre Magnot, ``Ambrose-Singer theorem on diffeological bundles and complete integrability of the KP equation'', \emph{Int.\ J.\ Geom.\ Meth.\ Mod.\ Phys.} \textbf{10} (2013), 31 pp.

\bibitem[Ma15]{Ma2015}
Jean-Pierre Magnot, \emph{$q$-Deformed Lax Equations and their Differential Geometric Background}, Lambert Academic Publishing, Saarbrucken, Germany, 2015. 

\bibitem[Ma16]{Ma2016}
Jean-Pierre Magnot, ``$\Diff(M)$-pseudo-differential operators and the geometry of non-linear grassmannians'', \emph{Mathematics} \textbf{4} (2016), 1.

%\bibitem{milnor1}
%John Milnor, ``Construction of univeral bundles, I'' \emph{Ann.\ of Math.\ (2)}, \textbf{63} (1956), 272--284.

\bibitem[Mi]{milnor2}
John Milnor, ``Construction of universal bundles, II'', \emph{Ann.\ of Math.\ (2)}, \textbf{63} (1956), 430--436.

\bibitem[Mo]{mostow}
Mark A.\ Mostow, ``The differentiable space structures of Milnor classifying spaces, simplicial complexes, and geometric realizations'', \emph{J. Differential Geom.} \textbf{14} (1979), 255--293.

\bibitem[Mu]{munkres}
James R.\ Munkres, \emph{Topology, 2nd Edition}, Prentice Hall, 2000.

\bibitem[N]{neeb}
Karl-Hermann Neeb, ``Towards a Lie theory for locally convex groups,'' \emph{Jpn.\ J.\ Math.} \textbf{1} (2006), 291--468.

\bibitem[O]{Om}
Hideki Omori, \emph{Infinite-Dimensional Lie Groups, Trans.\ Math.\ Monographs 158}, Amer.\ Math.\ Soc., 1997.

\bibitem[P]{pervova}
Ekaterina Pervova, ``Diffeological vector pseudo-bundles'', \emph{Topology Appl.} \textbf{202} (2016), 269--300.

\bibitem[R]{R}
Steven Rosenberg, ``Chern-Weil theory for certain infinite-dimensional Lie groups, Lie Groups: Structure, Actions, and Representations'', \emph{Progr. Math.} \textbf{306} (2013), Birkhäuser/Springer, New York, 355--380.

\bibitem[tD]{tD}
Tammo tom Dieck, \emph{Algebraic Topology, EMS Textbooks in Mathematics}, European Mathematical Society, (2008).

\bibitem[Wa]{watts-phd}
Jordan Watts, \emph{Diffeologies, Differential Spaces and Symplectic Geometry}, PhD thesis, University of Toronto, (2012).\\
\url{http://arxiv.org/abs/1208.3634}

\bibitem[We]{weinstein}
Alan Weinstein, ``Cohomology of symplectomorphism groups and critical values of Hamiltonians'', \emph{Math.\ Z.} \textbf{201} (1989), 75--82.

\end{thebibliography}
\end{document}